\documentclass[mathpazo]{cicp}
\usepackage{syntonly,indentfirst,mathrsfs,latexsym,float}
\usepackage{bm,amsmath,amsbsy,amsthm,amssymb,amsfonts}
\usepackage{array}
\usepackage{multirow}
\usepackage{mathrsfs}
\usepackage{type1cm}
\usepackage{wasysym}
\usepackage{cite,fancyhdr,dsfont,enumerate,graphicx,color,stmaryrd}
\usepackage[colorlinks,linkcolor=black,citecolor=blue]{hyperref}
\usepackage[capitalise, nosort]{cleveref}

\usepackage{indentfirst}
\usepackage{stmaryrd}
\usepackage{subfigure}
\usepackage{graphicx}
\usepackage[utf8]{inputenc}
\usepackage{pgf,tikz}
\usepackage{threeparttable}
\usetikzlibrary{arrows}
\numberwithin{equation}{section}
\newcommand{\tabcaption}{\def\@captype{table}\caption}
\newtheorem{lem}{Lemma}[section]
\newtheorem{thm}{Theorem}[section]

\newtheorem{rem}{Remark}[section]
\newtheorem{defi}{Definition}[section]

\newtheorem{exmp}{Example}[section]

\begin{document}
	\title{Robust globally divergence-free weak Galerkin finite element methods for  natural convection problems}
		\author[Han Y H et.~al]{Yihui Han, Xiaoping Xie\corrauth}
	\address{School of Mathematics, Sichuan University, Chengdu 610064, China}
	 \emails{{\tt 1335751459@qq.com} (Y. ~Han), {\tt xpxie@scu.edu.cn} (X. ~Xie)}

 \begin{abstract}
	This paper  proposes and analyzes   a class of  weak Galerkin (WG) finite element methods for  stationary  natural convection problems in  two  and three dimensions. We use piecewise polynomials of degrees $k, k-1,$ and $k$   $(k\geq 1)$ for the  velocity, pressure, and temperature approximations in the interior of elements, respectively,  and piecewise  polynomials of degrees $l, k, l $ $(l = k-1, k)$ for the numerical traces of velocity, pressure and temperature on the interfaces of elements. The methods yield globally divergence-free velocity solutions.  Well-posedness  of the discrete scheme is established,  optimal a priori error estimates are derived, and an unconditionally convergent iteration algorithm is presented. Numerical experiments  confirm the theoretical results and show   the robustness of the methods with respect to Rayleigh number.
\end{abstract}

\maketitle

\section{Introduction}
Let  $\mathbb{R}^d (d = 2, 3)$ be a polygonal or polyhedral domain with a polygonal or polyhedral subdomain $\Omega_f \subset \Omega$ and $\Omega_s := \Omega\setminus \Omega_f$, we consider the following  stationary  natural convection (or conduction-convection) problem: seek the velocity $\bm{u}=(u_1,u_2,\cdots, u_d)^T$, the pressure $p$, and the temperature $T$ such that
\begin{align}
\label{pb1}
\left \{
\begin{array}{rl}
-\Pr   \Delta \bm{u}+\nabla \cdot(\bm{u}\otimes \bm{u})+\nabla p-\Pr   Ra\bm{j}T = \bm{f}&   in \quad\Omega_f,\\
\nabla \cdot \bm{u}=0&  in\quad \Omega_f,\\
-\kappa \Delta T+\nabla \cdot(\bm{u} T) = g&  in\quad \Omega,\\
\bm{u}\equiv\bm{0} & in\quad \Omega_s\bigcup \partial\Omega_f,\\
T = 0& on\quad \partial\Omega ,
\end{array}
\right.
\end{align}
where  $\otimes$ is defined by  $\bm{u}\otimes\bm{v}=(u_iv_j)_{d\times d}$ for $\bm v=(v_1,v_2, \cdots,v_d)^T$, $\bm{j}$ is the vector of gravitational acceleration with $\bm{j}=(0,1)^T$ when $d=2$  and   $\bm{j}=(0,0,1)^T$ when $d=3$,  $\bm{f}\in[L^2(\Omega_f)]^d$, $g\in L^2(\Omega)$ are the forcing functions, and $ \Pr $, $Ra$ denote the Prandtl and Rayleigh numbers, respectively,. 

The  model problem \eqref{pb1}, arising both  in nature and in  engineering applications,  is  a coupled system of fluid flow, governed by the incompressible Navier-Stokes equations, and heat transfer, governed by the  energy equation.   Due to   its  practical significance, the development of efficient numerical methods  for   natural convection  has 
 attracted a great many of research efforts; see, e.g. 
%
%
\cite{Ben2011A},\cite{Boland1990An},\cite{boland1990error},\cite{c2001new},\cite{Mayne2000h-adaptive},\cite{de1983natural},
\cite{H1994An},\cite{JOHANNES2012A},\cite{Luo2003A},\cite{Luo2010An},\cite{manzari1999explicit},\cite{massarotti1998characteristic},\cite{Shi2009Nonconforming},\cite{Shi2010A},
\cite{Si2011A},
\cite{Zhang2014Error},\cite{Zhang2014A}. 
%
 In \cite{Boland1990An, boland1990error},    error estimates for  some finite element methods were derived in approximating 
stationary and non-stationary  natural convection problems. 
  \cite{Luo2003A, Luo2010An} applied   Petrov-Galerkin least squares mixed finite element  methods to discretize the problems. \cite{Shi2009Nonconforming,Shi2010A} developed a nonconforming mixed  element method and a  Petrov-Galerkin least squares nonconforming mixed  element method for the stationary problems. In \cite{zhang2015decoupled},   three kinds of decoupled two level finite element methods were presented. \cite{Zhang2014Error, Zhang2014A} applied  the variational multiscale method to solve the stationary and non-stationary  problems. 

In this paper, we consider a weak Galerkin (WG) finite element discretization of the  model problem \eqref{pb1}.
The WG method was first proposed and analyzed to solve second-order elliptic problems \cite{wang2013weak,Wang2014weakmixed}. It is designed by using a weakly defined gradient operator over functions with discontinuity, and then allows the use of totally discontinuous functions in the finite element procedure. 
Similar to the hybridized discontinuous Galerkin (HDG) method  \cite{Cockburn2009Unified}, the WG  method is of the property of local elimination of   unknowns defined in the interior of elements. 
 We note that in some special cases the WG method and the HDG method are equivalent  (cf. \cite{chen2016robust,chen-feng-xie2017, chen-xie2016}).  %
In \cite{chen2016robust},  a class of robust globally divergence-free weak Galerkin    methods for Stokes equations were developed, and then  were extended in \cite{Zheng2017A}  to solve   incompressible quasi-Newtonian Stokes equations.
We also refer to \cite{Chen-W-W-Y2015,Li-X2015,Deka2018wginterface,Li-X2016,Mu-W-Y-Z2015, Wang-W-Z-Z2016, Wang-Y2016, Zheng-X2017,Chen-C-X2018,zhaiqilong2018hyperbolic,wangruishu2018interface,zhangjiachuan2018WGNS} for some other developments and applications of   the WG method.

This paper aims to propose   a class of WG methods for the natural convection problems. The methods include as unknowns the  velocity, pressure, and temperature variables both  in
the interior of elements and  on the interfaces of elements. In the interior of elements,  we use piecewise polynomials of degrees $k, k-1,$ and $k$   $(k\geq 1)$ for the  velocity, pressure, and temperature approximations, respectively. On the interfaces of elements, we use  piecewise  polynomials of degrees $l, k, l $ $(l = k-1, k)$ for the numerical traces of velocity, pressure and temperature. The methods are shown to  yield globally divergence-free velocity approximations.

%

The rest of the paper is organized as follows. Section 2 introduces    the WG finite element scheme. Section 3 shows the  existence and  uniqueness of the discrete solution. Section 4 derives a priori error estimates. Section 5 discusses  the local elimination property  and the convergence of an iteration method for the WG   scheme. Finally,   Section 6 provides numerical examples to verify the  theoretical results. 

Throughout this paper, we use $a\apprle b$ $(a\apprge b)$ to denote $a\leq Cb$ $(a\geq Cb)$, where the constant C is positive independent of mesh size $h, h_K, h_e$ and the $\Pr$, $\kappa$ and Rayleigh number.

\section{WG finite element scheme}
\subsection{Notation}
For any bounded domain $D\in \mathbb{R}^s (s= d, d-1)$, let $H^m(D)$ and $H_0^m(D)$ denote the usual $m^{th}$-order Sobolev spaces on D, and  $\lVert \cdot\rVert_{m,D}, \lvert\cdot\rvert_{m,D}$ denote the norm and semi-norm on these spaces. We use  $(\cdot,\cdot)_{m,D}$ to denote the inner product of $H^m(D)$, with $(\cdot,\cdot)_D := (\cdot,\cdot)_{0,D}$. When $D = \Omega$, we set $\lVert\cdot\rVert_m := \lVert\cdot\rVert_{m,\Omega},\lvert\cdot\rvert_m :=\lvert\cdot\rvert_{m,\Omega}, $ and $(\cdot,\cdot) := (\cdot,\cdot)_\Omega
$. In particular, when $D \subset R^{d-1}$, we use $\langle\cdot,\cdot\rangle_D$ to replace $(\cdot,\cdot)_D$. For integer $k\geqslant0$,  $P_k(D)$ denotes the set of all polynomials on D with degree no more than $k$. We also need the following spaces:
\begin{center}
	$L_0^2(\Omega) :=\{q\in L^2(\Omega) :(q,1)=0\}$,\\
	$\bm{H}(div,D) :=\{\bm{v}\in \bm{L}^2(D) :\nabla\cdot\bm{v}\in L^2(D)\}.$
\end{center}

Let $\mathcal{T}_h^s$ and $\mathcal{T}_h^f$ be shape-regular simplicial decompositions of the subdomains $\Omega_s$ and $\Omega_f$, respectively. Then   $\mathcal{T}_h := \mathcal{T}_h^s\cup \mathcal{T}_h^f=\bigcup \{K\}$ is  a shape-regular simplicial decomposition of   $\Omega$. 
Let  $\varepsilon_h^s$ and  $\varepsilon_h^f$ be the sets of all edges (faces) of all elements in $\mathcal{T}_h^s$ and $\mathcal{T}_h^f$, respectively, and set $\varepsilon_h: = \varepsilon_h^s\cup \varepsilon_h^f=\bigcup\{e\}$.
 For any $K\in \mathcal{T}_h$, $e\in \varepsilon_h$,  we denote by $h_K$ and $h_e$ the diameters of $K$ and $e$, respectively, and set $h: = \max\limits_{K\in \mathcal{T}_h}h_K$. Let $\bm{n}_K$ and $\bm{n}_e$ be the outward unit normal vectors along the boundary $\partial K$ and $e$.
We denote by $\nabla_h$ and $\nabla_h\cdot$ the piecewise-defined gradient and divergence with respect to   $\mathcal{T}_h$. We also introduce the mesh-dependent inner products and mesh-dependent norms:
\begin{align*}
\langle u,v\rangle_{\partial\mathcal{T}_h} :=&\sum_{K\in \mathcal{T}_h}\langle u,v\rangle_{\partial K},\quad 
\lVert u\rVert_{0,\partial\mathcal{T}_h} :=\left(\sum_{K\in \mathcal{T}_h}\lVert u\rVert_{0,\partial K}^2\right)^{1/2} \\
(u,v)_{\mathcal{T}_h} :=&\sum_{K\in \mathcal{T}_h}(u,v)_{K},\quad\quad 
\lVert u\rVert_{0,\mathcal{T}_h} :=\left(\sum_{K\in \mathcal{T}_h}\lVert u\rVert_{0,K}^2\right)^{1/2}.
\end{align*}

\subsection{Weak problem}
We first introduce the space 
\begin{align*}
\bm{W} := \{\bm{v}\in [H_0^1(\Omega_f)]^d: \nabla\cdot \bm{v}=0\}
\end{align*}
and   the following  bilinear and  trilinear forms: for any $ \bm{u},\bm{v}, \bm{w}\in H_0^1(\Omega_f)$, $q\in L_0^2(\Omega_f)$, and $T,s\in H_0^1(\Omega)$,
\begin{align*}
a(\bm{u},\bm{v}) &:= \Pr  (\nabla\bm{u},\nabla\bm{v}),\quad
b(\bm{v},q) := (\nabla q,\bm{v}),\\
d(T,\bm{v}) &:= \Pr  Ra(\bm{j}T,\bm{v}),\quad
\overline{a}(T,s) := \kappa(\nabla T,\nabla s),\\
c(\bm{w};\bm{u},\bm{v}) &:= ((\bm{w}\cdot \nabla)\bm{u},\bm{v}), \quad
\overline{c}(\bm{w};T,s) := ((\bm{w}\cdot \nabla)T,s). 
\end{align*}
It is easy to see that, for $\bm u\in \bm W$, 
\begin{align*}
c(\bm{u};\bm{u},\bm{v}) &=  \frac{1}{2}(\nabla\cdot (\bm{u}\otimes\bm{u}),\bm{v})-\frac{1}{2}(\nabla\cdot (\bm{u}\otimes\bm{v}),\bm{u}),\\
\overline{c}(\bm{u};T,s) &= \frac{1}{2}(\nabla\cdot (\bm{u}T),s)-\frac{1}{2}(\nabla\cdot (\bm{u}s),T). 
\end{align*}
Then the variational problem of (\ref{pb1}) reads as follows:  seek $(\bm{u},p,T)\in \bm{W}\times L_0^2(\Omega_f)\times H_0^1(\Omega)$  such that
\begin{align}
\label{pbv}
\left \{
\begin{array}{rl}
A(\bm{u};\bm{u},\bm{v})+b(\bm{v},p)-b(\bm{u},q)-d(T,\bm{v})&= (\bm{f},\bm{v}), \forall (\bm{v},q)\in \bm{W}\times L_0^2(\Omega_f),\\
\overline{A}(\bm{u};T,s)&= (g,s), \forall s\in H_0^1(\Omega),
\end{array}
\right.
\end{align}
where 
\begin{align*}
A(\bm{u};\bm{u},\bm{v}) &:= a(\bm{u},\bm{v})+c(\bm{u};\bm{u},\bm{v}),\\
\overline{A}(\bm{u};T,s)&=\overline{a}(T,s)+\overline{c}(\bm{u};T,s).
\end{align*}

\begin{thm}\label{continuous exsit}
	\cite{boland1990error}For  $\bm{f}\in [H^{-1}(\Omega_f)]^d$ and $g\in H^{-1}(\Omega)$, the weak problem (\ref{pbv}) has at least one solution.
	In addition,  it admits a unique   solution $(\bm{u},p,T)\in \bm{W}\times L_0^2(\Omega_f)\times H_0^1(\Omega)$ if 
	\begin{equation*}\label{assump1}
	({\Pr}^{-1}\mathcal{N}Ra\kappa^{-1}+Ra\mathcal{M}\kappa^{-2})\lVert g\rVert_{-1}+\mathcal{N}{\Pr}^{-2}\lVert \bm{f}\rVert_{-1} < 1,	
	\end{equation*}
where \begin{align*}
\mathcal{N} := \sup\limits_{\bm{0}\neq\bm{w},\bm{u},\bm{v}\in \bm{W}}\frac{c(\bm{w};\bm{u},\bm{v})}{\lvert \bm{w}\rvert_1\lvert \bm{u}\rvert_1\lvert \bm{v}\rvert_1},\quad
\mathcal{M} := \sup\limits_{\substack{\bm{0}\neq\bm{w}\in \bm{W},\\ 0\neq T,s\in H_0^1(\Omega)}}\frac{\overline{c}(\bm{w};T,s)}{\lvert \bm{w}\rvert_1\lvert T\rvert_1\lvert s\rvert_1}.
\end{align*}

\end{thm}
%
%

In what follows,  we assume that the solution $(\bm{u},p,T)$ is unique and,
more precisely, there exists  a fixed  constant $\delta> 0$ such that
\begin{equation*}\label{assump2}
({\Pr}^{-1}\mathcal{N}Ra\kappa^{-1}+Ra\mathcal{M}\kappa^{-2})\lVert g\rVert_{-1}+\mathcal{N}{\Pr}^{-2}\lVert \bm{f}\rVert_{-1} < 1-\delta.
\end{equation*}

\subsection{Discrete weak operators}
In order to design a WG finite element scheme for the problem (1.1), we   introduce  the  discrete weak gradient operator $\nabla_{w,r}$ and  the discrete weak divergence operator $\nabla_{w,r}\cdot$ as follows. 
%

\begin{defi}
	For any   $K\in \mathcal{T}_h$ and $v\in \mathcal{V}(K):=\left\{\{v_0,v_b\}:v_0\in L^2(K),v_b\in H^{1/2}(\partial K)\right\}$,  the discrete weak gradient $\nabla_{w,r,K}v\in [P_r(K)]^d$ on $K$ is determined by the  equation
	\begin{equation*}\label{weak gradient}
	(\nabla_{w,r,K}v,\bm{\tau})_K= -(v_0,\nabla\cdot\bm{\tau})_K+\langle v_b,\bm{\tau} \cdot \bm{n}_K\rangle_{\partial K}\quad \forall \bm{\tau}\in [P_r(K)]^d.                           
	\end{equation*}
\end{defi}
Then we define the global discrete weak gradient operator $\nabla_{w,r}$   by
\begin{center}
	$\nabla_{w,r}|_K=\nabla_{w,r,K},\quad \forall K\in \mathcal{T}_h$.
\end{center}
For a vector $\bm{v}=(v_1,\cdots,v_d)^T\in [\mathcal{V}(K)]^d$, we define its  discrete weak gradient $\nabla_{w,r}\bm{v}$ by  
$$\nabla_{w,r}\bm{v}:=(\nabla_{w,r}v_1,\cdots,\nabla_{w,r}v_d)^T.$$

\begin{defi}
	For any  $K\in \mathcal{T}_h$ and $\bm{v}\in \mathcal{W}(K):=\left\{\{\bm{v}_0,\bm{v}_b\}:\bm{v}_0\in [L^2(K)]^d,\bm{v}_b\cdot\bm{n}_K\in H^{-1/2}(\partial K)\right\}$, the discrete weak divergence $\nabla_{w,r,K}\cdot \bm{v}\in P_r(K)$ is determined by the equation
	\begin{equation*}\label{weak divergence}
	(\nabla_{w,r,K}\cdot \bm{v},\tau)_K= -(\bm{v}_0,\nabla\tau)_K+\langle \bm{v}_b\cdot \bm{n}_K,\tau \rangle_{\partial K}\quad \forall \tau\in P_r(K).                           
	\end{equation*}
\end{defi}
Then we define the global discrete weak divergence operator $\nabla_{w,r}\cdot$ by
$$\nabla_{w,r}\cdot|_K=\nabla_{w,r,K}\cdot,\forall K\in \mathcal{T}_h.$$ 
For a tensor $\bm w=(\bm w_1,\cdots, \bm w_d)^T\in [\mathcal{W}(K)]^{d\times d}$ with $\bm w_i\in [\mathcal{W}(K)]^{d}$ for $i=1,\cdots,d$, we define its discrete weak divergence $\nabla_{w,r}\cdot \bm w$ by
$$ \nabla_{w,r}\cdot \bm w=( \nabla_{w,r}\cdot \bm w_1,\cdots,  \nabla_{w,r}\cdot \bm w_d)^T.$$

\subsection{WG finite element scheme}
For any $K\in \mathcal{T}_h,e\in \varepsilon_h$ and any integer $j\geq 0$, let $Q_j^0:L^2(K)\rightarrow P_j(K)$ and $Q^b_j:L^2(e)\rightarrow P_j(e)$ be the usual $L^2$ projection operators. We shall use $\bm{Q}_j^b$ to denote $Q^b_j$ for vector spaces.  

For any integer $k\geq 1$ and $l = k-1,k$, we introduce the following finite dimensional  spaces:
\begin{align*}
\bm{V}_h&=\{\bm{v}_h=\{\bm{v}_{h0},\bm{v}_{hb}\}:\bm{v}_{h0}|_K\in [P_k(K)]^d,\bm{v}_{hb}|_e\in [P_l(e)]^d,\forall K\in \mathcal{T}_h,\forall e\in \varepsilon_h\},\\
\bm{V}_h^0&=\{\bm{v}_h=\{\bm{v}_{h0},\bm{v}_{hb}\}\in \bm{V}_h:\bm{v}_{hb}|_{\partial \Omega_f}=0\},\\
Q_h&=\{q_h=\{q_{h0},q_{hb}\}:q_{h0}|_K\in P_{k-1}(K),q_{hb}|_e\in P_k(e),\forall K\in \mathcal{T}_h,\forall e\in \varepsilon_h\},\\
Q_h^0&=\{q_h=\{q_{h0},q_{hb}\}\in Q_h:q_{h0}\in L_0^2(\Omega_f)\},\\
S_h&=\{s_h=\{s_{h0},s_{hb}\}:s_{h0}|_K\in P_k(K),s_{hb}|_e\in P_l(e),\forall K\in \mathcal{T}_h,\forall e\in \varepsilon_h\},\\
S_h^0&=\{s_h=\{s_{h0},s_{hb}\}\in S_h:s_{hb}|_{\partial \Omega}=0\}.
\end{align*}

For  any $\bm{u}_h =\{\bm{u}_{h0},\bm{u}_{hb}\}, \bm{v}_h =\{\bm{v}_{h0},\bm{v}_{hb}\}\in \bm{V}_h^0$, $q_h = \{q_{h0},q_{hb}\}\in Q_h^0$, and $T_h = \{T_{h0},T_{hb}\}, s_h = \{s_{h0},s_{hb}\}\in S_h^0$, define the following bilinear and trilinear forms:
\begin{align*}
a_h(\bm{u}_h,\bm{v}_h) &:= \Pr  (\nabla_{w,m}\bm{u}_h,\nabla_{w,m}\bm{v}_h)+\Pr  \langle \tau(\bm{Q}_l^b\bm{u}_{h0}-\bm{u}_{hb}),\bm{Q}_l^b\bm{v}_{h0}-\bm{v}_{hb}\rangle_{\partial \mathcal{T}_h^f},\\
b_h(\bm{v}_h,q_h) &:= (\nabla_{w,k}q_h,\bm{v}_{h0}),\\
d_h(T_h,\bm{v}_h) &:= \Pr  Ra(\bm{j}T_{h0},\bm{v}_{h0}),\\
\overline{a}_h(T_h,s_h) &:= \kappa(\nabla_{w,m}T_h,\nabla_{w,m}s_h)+\kappa\langle \tau(Q_l^bT_{h0}-T_{hb}),Q_l^bs_{h0}-s_{hb}\rangle_{\partial \mathcal{T}_h},\\
c_h(\bm{w}_h;\bm{u}_h,\bm{v}_h) &:= \frac{1}{2}(\nabla_{w,k}\cdot (\bm{u}_h\otimes\bm{w}_h),\bm{v}_{h0})-\frac{1}{2}(\nabla_{w,k}\cdot (\bm{v}_h\otimes\bm{w}_h),\bm{u}_{h0}),\\
\overline{c}_h(\bm{u}_h;T_h,s_h) &:= \frac{1}{2}(\nabla_{w,k}\cdot (\bm{u}_hT_h),s_{h0})-\frac{1}{2}(\nabla_{w,k}\cdot (\bm{u}_hs_h),T_{h0}).
\end{align*}
It is easy to see that 
\begin{equation}
	 c_h(\bm{w}_h;\bm{v}_h,\bm{v}_h) =0,\quad 
	  \overline{c}_h(\bm{w}_h;s_h,s_h)  =0.\label{3356}
\end{equation}

The WG finite element scheme for (\ref{pb1}) is then given as follows: seek  $\bm{u}_h = \{\bm{u}_{h0},\bm{u}_{hb}\}\in \bm{V}_h^0$, $p_h = \{p_{h0},p_{hb}\}\in Q_h^0$, and $T_h =\{T_{h0},T_{hb}\}\in S_h^0$ such that
\begin{align}
\label{pb2}
\left \{
\begin{array}{rl}
A_h(\bm{u}_h;\bm{u}_h,\bm{v}_h)+b_h(\bm{v}_h,p_h)-b_h(\bm{u}_h,q_h)-d_h(T_h,\bm{v}_h)&= (\bm{f},\bm{v}_{h0}), \forall (\bm{v}_h,q_h)\in \bm{V}_h^0\times Q_h^0,\\
\overline{A}_h(\bm{u}_h;T_h,s_h)&= (g,s_{h0}), \forall s_h\in S_h^0,
\end{array}
\right.
\end{align}
where 
\begin{align}
A_h(\bm{w}_h;\bm{u}_h,\bm{v}_h) &:= a_h(\bm{u}_h,\bm{v}_h)+c_h(\bm{w}_h;\bm{u}_h,\bm{v}_h),\\
\overline{A}_h(\bm{u}_h;T_h,s_h)&:=\overline{a}_h(T_h,s_h)+\overline{c}_h(\bm{u}_h;T_h,s_h),\label{224}
\end{align}
$\tau|_{\partial T}=h_T^{-1}$, and m is an integer with $k-1\leq m\leq l$.

\begin{rem}
	It's easy to show that the scheme (\ref{pb2}) yields globally divergence-free velocity approximation $\bm{u}_{h0}$. In fact, let $K_1,K_2\in \mathcal{T}_h$ be any two adjacent elements with a common face $e$, introduce a function $r_{hb}\in L^2(\varepsilon_h)$ with
	\begin{equation*} 
	r_{hb}|_e=\left\{
	\begin{array}{l}
	-(\bm{u}_{h0}\cdot \bm{n}_e)|_{K_1\bigcap e}-(\bm{u}_{h0}\cdot \bm{n}_e)|_{K_2\bigcap e},\quad \forall e\in \varepsilon_h^f/\partial \Omega_f,\\
	0,\quad \forall e\in \partial\Omega_f,
	\end{array}
	\right.
	\end{equation*}
	and set $c_0:=\frac{1}{|\Omega_f|} \int_{\Omega_f} \nabla_h\cdot \bm{u}_{h0}d\bm{x}$. Then, taking $(\bm{v}_{h0},\bm{v}_{hb},q_{h0},q_{hb},s_{h0},s_{hb})=(\bm{0},\bm{0},\nabla_h\cdot \bm{u}_{h0}-c_0,r_{hb}-c_0,0,0) $ 
	 in (\ref{pb2}) yields
	\begin{equation*}
	\lVert \nabla_h\cdot \bm{u}_{h0}\rVert_0^2 +\sum_{e\in \varepsilon_h^f/\partial \Omega_f}\lVert (\bm{u}_{h0}\cdot \bm{n}_e)|_{K_1}+(\bm{u}_{h0}\cdot \bm{n}_e)|_{K_2}\rVert_{0,e}^2=0.
	\end{equation*}
	This indicates $\bm{u}_{h0}\in\bm{H}(div,\Omega_f)$ and $\nabla_h\cdot \bm{u}_{h0}=\nabla\cdot \bm{u}_{h0}=0$, i.e.  the velocity approximation $\bm{u}_{h0}$ is globally divergence-free  in a pointwise sense.
\end{rem}

\section{Well-posedness of the discrete scheme}
\subsection{Some basic results}
For the projections $Q_j^0$ and $Q_j^b$ with $j\geq 0$, the following stability and approximation results are standard.

\begin{lem}(\cite{shi2013finite})
\label{ineq} 
	Let $s$ be an integer with $1\leq s\leq j+1$. Then we have, for any $K\in \mathcal{T}_h$ and $e\in \varepsilon_h$, 
	\begin{align*}
	\lVert v-Q_j^0v\rVert_{0,K}+h_K\lvert v-Q_j^0v\rvert_{1,K}&\apprle h_K^s\lvert v\rvert_{s,K},\forall v\in H^s(K),\\
	\lVert v-Q_j^0v\rVert_{0,\partial K}&\apprle h_K^{s-1/2}\lvert v\rvert_{s,K},\forall v\in H^s(K),\\
	\lVert v-Q_j^bv\rVert_{0,\partial K}&\apprle h_K^{s-1/2}\lvert v\rvert_{s,K},\forall v\in H^s(K),\\
	\lVert Q_j^0v\rVert_{0,K}&\leq \lVert v\rVert_{0,K},\forall v\in L^2(K),\\
	\lVert Q_j^bv\rVert_{0,e}&\leq \lVert v\rVert_{0,e},\forall v\in L^2(e).
	\end{align*}
\end{lem}

By using the trace theorem, the inverse inequality, and scaling arguments metioned in {\cite{shi2013finite}}, we can get the following lemma. 
\begin{lem}\label{trace inequality}
	For all $K\in \mathcal{T}_h$, $w\in H^1(K)$, and $1\leq \tilde{q}\leq \infty$, we have
	\begin{equation*}
	\lVert w\rVert_{0,\tilde{q},\partial K}\apprle h_K^{-\frac{1}{\tilde{q}}}\lVert w\rVert_{0,\tilde{q},K} + h_K^{1-\frac{1}{\tilde{q}}}\lvert w\rvert_{1,\tilde{q},K}, 
	\end{equation*}
	In particular, for all   $w\in  P_k(K)$, 
    \begin{equation*}
    \lVert w\rVert_{0,\tilde{q},\partial K}\apprle h_K^{-\frac{1}{\tilde{q}}}\lVert w\rVert_{0,\tilde{q},K}.
    \end{equation*}
\end{lem}

%

%

\begin{lem}(\cite{chen2016robust}) \label{lem3.3}
	Let $0\leq k-1\leq m\leq l\leq k$. For all $K\in \mathcal{T}_h$ and $\bm{v}_h=\{\bm{v}_{h0},\bm{v}_{hb}\}\in [P_k(K)]^d\times [P_l(\partial K)]^d$, the following estimates hold:
	\begin{align}
	\lVert \nabla\bm{v}_{h0}\rVert_{0,K}&\apprle \lVert \nabla_{w,m}\bm{v}_h\rVert_{0,K}+h_K^{-1/2}\lVert \bm{Q}_l^b\bm{v}_{h0}-\bm{v}_{hb}\rVert_{0,\partial K},\label{equivalentu1}\\
	\lVert \nabla_{w,m}\bm{v}_h\rVert_{0,K}&\apprle \lVert \nabla\bm{v}_{h0}\rVert_{0,K}+h_K^{-1/2}\lVert \bm{Q}_l^b\bm{v}_{h0}-\bm{v}_{hb}\rVert_{0,\partial K},\label{equivalentu2}
	\end{align}
	\end{lem}

We introduce the following semi-norms: for any $(\bm{v}_h,q_h,s_h)\in \bm{V}_h\times Q_h\times S_h$, 
\begin{align*}
\interleave\bm{v}_h\interleave^2:=& \lVert\nabla_{w,m}\bm{v}_h\rVert_0^2+ \lVert\tau^{1/2}(\bm{Q}_l^b\bm{v}_{h0}-\bm{v}_{hb})\rVert_{0,\partial\mathcal{T}_h^f}^2,\\
\lVert q_h\rVert^2:=&\lVert q_{h0}\rVert_0^2+\sum_{K\in\mathcal{T}_h^f}\lVert\nabla_{w,k}q_h\rVert_{0,K}^2,\\
\interleave s_h\interleave^2:=& \lVert\nabla_{w,m} s_h\rVert_0^2+ \lVert\tau^{1/2}(Q_l^bs_{h0}-s_{hb})\rVert_{0,\partial\mathcal{T}_h}^2.
\end{align*}
Here we recall that  $\tau|_{\partial K}=h_K^{-1}$. It is easy to see that the above three semi-norms are norms on $\bm{V}_h^0$, $Q_h^0$ and $S_h^0$, respectively (cf. \cite{chen2016robust}).
In addition, from the lemma above it follows
	\begin{equation}\label{1seminorm v}
	\lVert \nabla_h\bm{v}_{h0}\rVert_0\apprle \interleave \bm{v}_h\interleave,\quad \forall \bm{v}_h\in \bm{V}_h^0.
	\end{equation}
	
\begin{rem} We note that the   estimates \eqref{equivalentu1}, \eqref{equivalentu2}, and \eqref{1seminorm v} also hold  for all $s_h\in S_h^0$ due to the fact that  $s_h|_K=\{s_{h0},s_{hb}\}\in P_k(K)\times P_l(\partial K)$.
\end{rem}
	
\begin{lem}(\cite{karakashian2007convergence})
	For all $s_{h0}\in S_{h0} = \{s_{h0}:s_{h0}|_K\in P_k(K),\forall K\in \mathcal{T}_h\}$, there exists an interpolation 
	$I_ks_{h0}\in S_{h0}\cap H_0^1(\Omega)$ such that
	\begin{align*}
	\sum_{K\in \mathcal{T}_h}\lVert s_{h0}-I_ks_{h0}\rVert_{0,K}^2&\apprle \sum_{e\in \varepsilon_h} h_e\lVert [s_{h0}]\rVert_{0,e}^2,\\
	\sum_{K\in \mathcal{T}_h}\lVert \nabla(s_{h0}-I_ks_{h0})\rVert_{0,K}^2&\apprle \sum_{e\in \varepsilon_h}h_e^{-1}\lVert [s_{h0}]\rVert_{0,e}^2.
	\end{align*}
	\end{lem}
From this lemma it follows that, for all $\bm{v}_{h0}\in \bm{V}_{h0} = \{\bm{v}_{h0}:\bm{v}_{h0}|_K\in P_k(K),\forall K\in \mathcal{T}_h^f\}$, there exists an interpolation  $\bm{I}_k\bm{v}_{h0}\in \bm{V}_{h0}\cap [H_0^1(\Omega)]^d$ such that
	\begin{align}
	\sum_{K\in \mathcal{T}_h}\lVert \bm{v}_{h0}-\bm{I}_k\bm{v}_{h0}\rVert_{0,K}^2&\apprle \sum_{e\in \varepsilon_h} h_e\lVert [\bm{v}_{h0}]\rVert_{0,e}^2,\label{bmI_k}\\
	\sum_{K\in \mathcal{T}_h}\lVert \nabla(\bm{v}_{h0}-\bm{I}_k\bm{v}_{h0})\rVert_{0,K}^2&\apprle \sum_{e\in \varepsilon_h}h_e^{-1}\lVert [\bm{v}_{h0}]\rVert_{0,e}^2.\label{bmI_ksemi}
	\end{align}


\begin{lem}\label{discrete sobolev embedding}
	For all $\bm{v}_h\in \bm{V}_h^0$ and $s_h\in S_h^0$, we have  
	\begin{align}
	\lVert \bm{v}_{h0}\rVert_{0,\tilde{q}}&\leq  C_{\tilde{q}1}\interleave \bm{v}_h\interleave,\label{ineq 11}\\
	\lVert s_{h0}\rVert_{0,\tilde{q}}&\leq  C_{\tilde{q}2}\interleave s_h\interleave,\label{ineq 12}
	\end{align}
	where  $2\leq \tilde{q}< \infty$ when $d = 2$, $2 \leq \tilde{q}\leq 6$   when $d = 3$,  and $C_{\tilde{q}1}$, $C_{\tilde{q}2}$ are positive constants only depending on $\tilde{q}$.
    \end{lem}
\begin{proof}
	 For all $\bm{v}_h\in \bm{V}_h^0$, we apply  the Sobolev embedding theorem and Poinc\'{a}re inequality to get
	\begin{equation}\label{318}
	\lVert \bm{I}_k\bm{v}_{h0}\rVert_{0,\tilde{q}}\apprle \lVert \bm{I}_k\bm{v}_{h0}\rVert_1 \apprle \lVert \nabla\bm{I}_k\bm{v}_{h0}\rVert_0.
	\end{equation}
	From (\ref{bmI_ksemi}), (\ref{1seminorm v}), 
	 the definition of $ \interleave\cdot \interleave$, and the projection property of $\bm{Q}_l^b$, it follows
	\begin{equation}\label{319}
	\begin{aligned}
	\lVert \nabla\bm{I}_k\bm{v}_{h0}\rVert_0 &\apprle \lVert \nabla_h \bm{v}_{h0}\rVert_0 + \big(\sum_{e\in \varepsilon_h}h_e^{-1}\lVert [\bm{v}_{h0}]\rVert_{0,e}^2\big)^{\frac{1}{2}}\\
	&\apprle \interleave \bm{v}_h\interleave + \big(\sum_{e\in \varepsilon_h}h_e^{-1}\lVert [\bm{v}_{h0}-\bm{v}_{hb}]\rVert_{0,e}^2\big)^{\frac{1}{2}}\\
	&\apprle \interleave \bm{v}_h\interleave.
	\end{aligned}
	\end{equation}
Using  the Sobolev embedding theorem and  the inverse inequality once again, by the properties of the projection-mean operator (\cite{shi2013finite}) $\bm{\Pi_h}:\bm{V}_{h0} = \{\bm{v}_{h0}:\bm{v}_{h0}|_K\in P_k(K),\forall K\in \mathcal{T}_h^f\}\rightarrow W^{1,2}(\Omega_f)\cap W^{0,\tilde{q}}(\Omega_f)$ and the fact that $2\leq \tilde{q}< \infty$ when $d = 2$ and  $2 \leq \tilde{q}\leq 6$   when $d = 3$, we have
	\begin{equation*}
	\begin{aligned}
	\lVert \bm{v}_{h0}-\bm{I}_k\bm{v}_{h0}\rVert_{0,\tilde{q}} &\leq \lVert \bm{v}_{h0}-\bm{\Pi_h}\bm{v}_{h0}\rVert_{0,\tilde{q}}+\lVert \bm{\Pi_h}\bm{v}_{h0}-\bm{I}_k\bm{v}_{h0}\rVert_{0,\tilde{q}}\\
	&\apprle h\lVert \nabla_h\bm{v}_{h0}\rVert_{0,\tilde{q}} +\lVert \bm{\Pi_h}\bm{v}_{h0}-\bm{I}_k\bm{v}_{h0}\rVert_{1,2} \\
	&\apprle h^{1-(\frac{d}{2}-\frac{d}{\tilde{q}})}\lVert \nabla_h\bm{v}_{h0}\rVert_{0,2} + \lVert \bm{v}_{h0}-\bm{\Pi_h}\bm{v}_{h0}\rVert_{1,2}+\lVert \bm{v}_{h0}-\bm{I}_k\bm{v}_{h0}\rVert_{1,2}\\
	&\apprle \lVert \nabla_h \bm{v}_{h0}\rVert_0+\lVert \nabla_h (\bm{v}_{h0}-\bm{I}_k\bm{v}_{h0})\rVert_0\\
	&\apprle \interleave \bm{v}_h\interleave,
	\end{aligned}
	\end{equation*}	
which, together with \eqref{318} and \eqref{319}, yields the desired estimate (\ref{ineq 11}).

Similarly, we can obtain   (\ref{ineq 12}).  This finishes the proof.
\end{proof}

For any nonnegative integer $j$ and any $K\in \mathcal{T}_h$, we introduce the local Raviart-Thomas(RT) element space
\begin{equation*}
\bm{RT}_j(K)=[P_j(K)]^d+\bm{x}P_j(K).
\end{equation*}

Lemmas \ref{RT3.6}-\ref{RT38} show some properties of the $RT$ projection which can be founded in ({\cite{brezzi2008mixed}}.Page 9-10).

\begin{lem}\label{RT3.6}
	For any $\bm{v}_{h0}\in \bm{RT}_j(K)$, $\nabla\cdot \bm{v}_{h0}|_K=0$ implies $\bm{v}_{h0}\in [P_j(K)]^d$.
\end{lem}

\begin{lem}\label{RT}
	For any $K\in \mathcal{T}_h$ and $\bm{v}\in [H^1(K)]^d$, there exists a unique $\bm{P}_j^{RT}\bm{v}\in \bm{RT}_j(K)$ such that
	\begin{align}
	\langle\bm{P}_j^{RT}\bm{v}\cdot \bm{n}_e,w_j\rangle_e&= \langle\bm{v}\cdot \bm{n}_e,w_j\rangle_e,\quad \forall w_j\in P_j(e),e\in \partial K,\label{RT1}\\
	(\bm{P}_j^{RT}\bm{v},\bm{w}_{j-1})_K&= (\bm{v},\bm{w}_{j-1})_K,\quad \forall \bm{w}_{j-1}\in [P_{j-1}(K)]^d\label{RT2}. 
	\end{align}
	If $j=0$, $\bm{P}_j^{RT}\bm{v}$ is determined only by (\ref{RT1}). Moreover, the following approximation holds:
	\begin{equation*}\label{RT3}
	\lVert \bm{v}-\bm{P}_j^{RT}\bm{v}\rVert_{0,K}\apprle h_K^r\lvert \bm{v}\rvert_{r,K},\quad \forall 1\leq r\leq j+1,\forall \bm{v}\in[H^r(K)]^d.
	\end{equation*}
\end{lem}

\begin{lem}\label{RT38}
	The operator $\bm{P}_j^{RT} $ defined in Lemma \ref{RT} satisfies
	\begin{equation*}\label{RTdiv}
	(\nabla\cdot \bm{P}_j^{RT}\bm{v},q_h)_K= (\nabla\cdot \bm{v},q_h)_K,\quad \forall \bm{v}\in [H^1(K)]^d, q_h\in P_j(K),K\in \mathcal{T}_h.
	\end{equation*}
\end{lem}


\begin{lem}(\cite{chen2016robust})\label{commutativity}
	It holds the following commutativity properties:
	\begin{align}
	\nabla_{w,m}\{\bm{P}_k^{RT}\bm{v},\bm{Q}_l^b\bm{v}\}&=\bm{Q}_m^0(\nabla\bm{v}),\quad \forall \bm{v}\in [H^1(\Omega_f)]^d. \label{com1}\\
	\nabla_{w,k}\{Q_{k-1}^0q,Q_k^bq\}&=\bm{Q}_k^0(\nabla q),\quad \forall q\in H^1(\Omega_f)\label{com2},\\
	\nabla_{w,m}\{Q_k^0 s,Q_l^b s\}&=\bm{Q}_m^0(\nabla s),\quad \forall s\in H^1(\Omega). \label{com3}
	\end{align}
\end{lem}
\subsection{Stability conditions}
\begin{lem}\label{boundedness}
	For any $\bm{u}_h,\bm{v}_h \in \bm{V}_h$, and $T_h,s_h \in S_h$,   the following inequalities hold:
	\begin{align}
	a_h(\bm{u}_h,\bm{v}_h) &\apprle \Pr\interleave \bm{u}_h\interleave \cdot\interleave \bm{v}_h\interleave,\label{continuous result ah} \\
	a_h(\bm{v}_h,\bm{v}_h) &= \Pr\interleave \bm{v}_h\interleave^2,\label{coercivity ah} \\
	\overline{a}_h(T_h,s_h) &\apprle \kappa\interleave T_h\interleave \cdot\interleave s_h\interleave,\label{continuous result ahbar} \\
	\overline{a}_h(s_h,s_h) &= \kappa\interleave s_h\interleave^2,\label{coercivity ahbar} \\
	 c_h(\bm{w}_h;\bm{u}_h,\bm{v}_h) &\apprle \interleave \bm{w}_h\interleave \cdot \interleave \bm{u}_h\interleave \cdot \interleave \bm{v}_h\interleave ,\label{continuous result ch} \\
 \overline{c}_h(\bm{w}_h;T_h,s_h)  &\apprle \interleave \bm{w}_h\interleave \cdot \interleave T_h\interleave \cdot \interleave s_h\interleave ,\label{continuous result chbar} \\
	d_h(T_h,\bm{v}_h)&\apprle \Pr Ra \interleave T_h\interleave \cdot\interleave \bm{v}_h\interleave. \label{continuous result dh}
	\end{align}
\end{lem}
\begin{proof}
	From the definitions of $a_h(\cdot,\cdot), \overline{a}_h(\cdot,\cdot),c_h(\cdot;\cdot,\cdot),\overline{c}_h(\cdot;\cdot,\cdot), d_h(\cdot,\cdot)$, Cauchy-Schwarz inequality and Lemma \ref{discrete sobolev embedding}, we can easily get   (\ref{continuous result ah}),(\ref{continuous result ahbar}), and (\ref{continuous result dh}).
	
	For all $\bm{u}_h,\bm{v}_h \in \bm{V}_h$, by the definition of $\nabla_{w,k}\cdot$ we have
	\begin{align*}
	2  c_h(\bm{w}_h;\bm{u}_h,\bm{v}_h) = &(\bm{v}_{h0}\otimes\bm{w}_{h0},\nabla_h\bm{u}_{h0})-(\bm{u}_{h0}\otimes\bm{w}_{h0},\nabla_h\bm{v}_{h0})  \\
	&\quad - \langle \bm{v}_{hb}\otimes\bm{w}_{hb}\bm{n},\bm{u}_{h0}\rangle_{\partial\mathcal{T}_h^f} +\langle \bm{u}_{hb}\otimes\bm{w}_{hb}\bm{n},\bm{v}_{h0}\rangle_{\partial\mathcal{T}_h^f}\\
	= &\big((\bm{v}_{h0}\otimes\bm{w}_{h0},\nabla_h\bm{u}_{h0})-(\bm{u}_{h0}\otimes\bm{w}_{h0},\nabla_h\bm{v}_{h0})\big)  \\
	&+  \langle (\bm{u}_{h0}-\bm{u}_{hb})\otimes(\bm{w}_{h0}-\bm{w}_{hb})\bm{n},\bm{v}_{h0}\rangle \\
	&-  \langle (\bm{u}_{h0}-\bm{u}_{hb})\otimes\bm{w}_{h0}\bm{n},\bm{v}_{h0}\rangle \\
    &- \langle (\bm{v}_{h0}-\bm{v}_{hb})\otimes(\bm{w}_{h0}-\bm{w}_{hb})\bm{n},\bm{u}_{h0}\rangle  \\
    &+  \langle (\bm{v}_{h0}-\bm{v}_{hb})\otimes\bm{w}_{h0}\bm{n},\bm{u}_{h0}\rangle   \\
    	=:&  \sum_{i=1}^{5}R_i.
	\end{align*}
    In light of  H\"{o}lder's inequality and Lemma \ref{discrete sobolev embedding}, we obtain
    \begin{align*}
    \lvert R_1 \rvert &\leq \lVert \bm{v}_{h0}\rVert_{0,4}\lVert \bm{w}_{h0}\rVert_{0,4}\lVert \nabla_h\bm{u}_{h0}\rVert_{0,2} +\lVert \bm{u}_{h0}\rVert_{0,4}\lVert \bm{w}_{h0}\rVert_{0,4}\lVert \nabla_h\bm{v}_{h0}\rVert_{0,2} \\
    &\apprle \interleave \bm{w}_h\interleave \cdot \interleave \bm{u}_h\interleave \cdot \interleave \bm{v}_h\interleave.
    \end{align*}
   From  H\"{o}lder's inequality, Lemma \ref{trace inequality}, Lemma \ref{discrete sobolev embedding},  and the inverse inequality,  it follows
    \begin{align*}
    |R_2| &\leq \sum_{K\in\mathcal{T}_h^f}\lVert \bm{w}_{h0}-\bm{w}_{hb}\rVert_{0,3,\partial K}\lVert \bm{u}_{h0}-\bm{u}_{hb}\rVert_{0,2,\partial K}\lVert \bm{v}_{h0}\rVert_{0,6,\partial K} \\
    &\leq \sum_{K\in\mathcal{T}_h^f}h_K^{-\frac{d-1}{6}}\lVert \bm{w}_{h0}-\bm{w}_{hb}\rVert_{0,2,\partial K}\lVert \bm{u}_{h0}-\bm{u}_{hb}\rVert_{0,2,\partial K}h_K^{-\frac{1}{6}}\lVert \bm{v}_{h0}\rVert_{0,6, K} \\
    &\leq \sum_{K\in\mathcal{T}_h^f}h_K^{-\frac{1}{2}}\lVert \bm{w}_{h0}-\bm{w}_{hb}\rVert_{0,2,\partial K}h_K^{-\frac{1}{2}}\lVert \bm{u}_{h0}-\bm{u}_{hb}\rVert_{0,2,\partial K}h_K^{1-\frac{d}{6}}\lVert \bm{v}_{h0}\rVert_{0,6, K} \\
   &\apprle \interleave \bm{w}_h\interleave \cdot \interleave \bm{u}_h\interleave \cdot \interleave \bm{v}_h\interleave
   \end{align*}
    and
    \begin{align*}
    |R_3| &\leq \sum_{K\in\mathcal{T}_h^f}\lVert \bm{w}_{h0}\rVert_{0,4,\partial K}\lVert \bm{u}_{h0}-\bm{u}_{hb}\rVert_{0,2,\partial K}\lVert \bm{v}_{h0}\rVert_{0,4,\partial K} \\
    &\leq \sum_{K\in\mathcal{T}_h^f}h_K^{-\frac{1}{4}}\lVert \bm{w}_{h0}\rVert_{0,4,K}\lVert \bm{u}_{h0}-\bm{u}_{hb}\rVert_{0,2,\partial K}h_K^{-\frac{1}{4}}\lVert \bm{v}_{h0}\rVert_{0,4,K} \\
    &\leq \sum_{K\in\mathcal{T}_h^f}\lVert \bm{w}_{h0}\rVert_{0,4,K}h_K^{-\frac{1}{2}}\lVert \bm{u}_{h0}-\bm{u}_{hb}\rVert_{0,2,\partial K}\lVert \bm{v}_{h0}\rVert_{0,4,K} \\
    &\apprle \interleave \bm{w}_h\interleave \cdot \interleave \bm{u}_h\interleave \cdot \interleave \bm{v}_h\interleave.
    \end{align*}
    Similarly, we can get 
     \begin{align*}
    |R_4|+|R_5| &\apprle \interleave \bm{w}_h\interleave \cdot \interleave \bm{u}_h\interleave \cdot \interleave \bm{v}_h\interleave.
    \end{align*}
       As a result,    the estimate (\ref{continuous result ch}) holds. 
       
       The estimate (\ref{continuous result chbar}) follows similarly.
\end{proof}

By \eqref{3356}, Lemma \ref{boundedness}, and the definitions of the trilinear forms $A_h(\cdot;\cdot,\cdot)$ and $\overline A_h(\cdot;\cdot,\cdot)$, we easily get the following continuity and coercivity results.

\begin{lem}\label{lem31}
	For any $\bm{w}_h,\bm{u}_h,\bm{v}_h \in \bm{V}_h, T_h,s_h \in S_h$, it holds
	\begin{align}
	A_h(\bm{w}_h;\bm{u}_h,\bm{v}_h) &\apprle (\Pr+\interleave \bm{w}_h\interleave)\interleave \bm{u}_h\interleave \cdot\interleave \bm{v}_h\interleave,\\
	\overline{A}_h(\bm{w}_h;T_h,s_h) &\apprle (\kappa +\interleave \bm{w}_h\interleave)\interleave T_h\interleave \interleave s_h\interleave,\\
	A_h(\bm{v}_h;\bm{v}_h,\bm{v}_h) &= \Pr \interleave \bm{v}_h\interleave^2,\\
	\overline{A}_h(\bm{v}_h;s_h,s_h) &= \kappa \interleave s_h\interleave^2 .\label{Core-Bar-A}
	\end{align}
\end{lem}

By following the same routine as in the proof of (\cite[Theorem 3.1]{chen2016robust}), we can obtain the following  inf-sup inequality.
\begin{lem}\label{infsup}
	 For any $(\bm{v}_h,q_h)\in \bm{V}_h^0\times Q_h^0$, it holds
	\begin{equation*}
	\sup\limits_{\bm{v}_h\in \bm{V}_h^0} \frac{b_h(\bm{v}_h,q_h)}{\interleave \bm{v}_h\interleave} \apprge \lVert q_h\rVert.
	\end{equation*}
\end{lem}

\subsection{Existence and uniqueness results}

We define a space 
\begin{equation*}
\bm{W}_h := \{\bm{w}_h\in \bm{V}_h^0:b_h(\bm{w}_h,q_h) = 0,\forall q_h\in Q_h^0\},
\end{equation*}
and introduce the following discretization problem: seek $(\bm{u_h},T_h)\in \bm{W}_h\times S_h^0$
\begin{align}\label{pb3}
\left \{
\begin{array}{rl}
A_h(\bm{u}_h;\bm{u}_h,\bm{v}_h)-d_h(T_h,\bm{v}_h) &= (\bm{f},\bm{v}_{h0}), \forall \bm{v}_h\in \bm{W}_h,\\
\overline{A}_h(\bm{u}_h;T_h,s_h) &= (g,s_{h0}), \forall s_h\in S_h^0.
\end{array}
\right.
\end{align}
It is easy to see that, by Lemma \ref{infsup} and  the theory of mixed finite element methods \cite{brezzi2008mixed},  the following conclusion holds.
\begin{lem}\label{lem313}
	The problems (\ref{pb2}) and (\ref{pb3}) are equivalent in the sense that (i) and (ii) hold:
	
	(i)  if $(\bm{u}_h,p_h,T_h)\in \bm{V}_h^0\times Q_h^0\times S_h^0$ is the solution to the problem (\ref{pb2}), then   $(\bm{u_h},T_h)$ 
	is the solution to  the problem \eqref{pb3};
	
	(ii)  if $(\bm{u_h},T_h)\in \bm{W}_h\times S_h^0$ is  the solution to  the problem \eqref{pb3}, then $(\bm{u}_h,p_h,T_h)$ 
	is the solution to the problem (\ref{pb2}), where $p_h\in Q_h^0$ is determined by 
	\begin{equation*}
	 b_h(\bm{v}_h,p_h)= (\bm{f},\bm{v}_{h0})-A_h(\bm{u}_h;\bm{u}_h,\bm{v}_h)+d_h(T_h,\bm{v}_h), \forall \bm{v}_h \in \bm{V}_h^0.
	 \end{equation*}
\end{lem}

 In what follows we shall discuss the   existence and uniqueness of the solution to the problem (\ref{pb3}). To this end, we set 
\begin{align}
\mathcal{N}_h &:= \sup\limits_{0\neq\bm{w}_h,\bm{u}_h,\bm{v}_h\in \bm{W}_h}\frac{c_h(\bm{w}_h;\bm{u}_h,\bm{v}_h)}{\interleave \bm{w}_h\interleave\cdot \interleave \bm{u}_h\interleave\cdot\interleave \bm{v}_h\interleave},\label{Nh}\\
\mathcal{M}_h &:= \sup\limits_{\substack{0\neq\bm{w}_h\in \bm{W}_h,\\ 0\neq T,s\in S_h^0}}\frac{\overline{c}_h(\bm{w}_h;T_h,s_h)}{\interleave \bm{w}_h\interleave\cdot \interleave T_h\interleave\cdot\interleave s_h\interleave},\label{Mh}\\
\lVert \bm{f}\rVert_h& := \sup\limits_{0\neq\bm{v}_h\in \bm{W}_h}\frac{(\bm{f},\bm{v}_{h0})}{\interleave \bm{v}_h\interleave},\label{f_h}\\
\lVert g\rVert_h &:= \sup\limits_{0\neq s_h\in S_h^0}\frac{(g,s_{h0})}{\interleave s_h\interleave}.\label{g_h}
\end{align}
From Lemma \ref{boundedness} we easily know  that  $\mathcal{N}_h, \mathcal{M}_h$ are bounded from above by a positive constant independent of the mesh size $h$.

\begin{thm} \label{exist}
The problem (\ref{pb3}) admits 
	  at least one solution $(\bm{u}_h,T_h)\in \bm{W}_h \times S_h^0$.
\end{thm}
\begin{proof}
	First, by Lemma \ref{lem31} it is easy to see that, for a given $\bm{u}_h\in \bm{W}_h$, the bilinear form 
	$\overline{A}_h(\bm{u}_h;\cdot,\cdot) 
	$ is   continuous and coercive  on $S_h^0\times S_h^0$. Hence, by   Lax-Milgram theorem there is a unique  $T_h\in S_h^0$ such that the second equation of (\ref{pb3}) holds.
	
	Define a mapping $F: \bm{W}_h\rightarrow S_h^0$ by $F(\bm{u}_h) = T_h$. Then the thing left is  to show that there exists at least one $\bm{u}_h\in \bm{W}_h$ such that 
	\begin{equation}\label{pb4}
	A_h(\bm{u}_h;\bm{u}_h,\bm{v}_h)=a_h(\bm{u}_h,\bm{v}_h) + c_h(\bm{u}_h;\bm{u}_h,\bm{v}_h) = d_h(F(\bm{u}_h),\bm{v}_h) + (\bm{f},\bm{v}_{h0}), \forall \bm{v}_h\in \bm{W}_h.
	\end{equation}
		
Take $s_h = T_h$ in the second equation of (\ref{pb3}),  and apply  (\ref{g_h}) and \eqref{Core-Bar-A} to   get
	\begin{equation*}
	\kappa \interleave T_h\interleave^2 = (g,T_{h0})\leq \lVert g\rVert_h\cdot \interleave T_h\interleave,
	\end{equation*}
	which yields 
	\begin {equation}\label{regular estimates Th}
	\interleave F(\bm{u}_h)\interleave = \interleave T_h\interleave \leq \kappa^{-1}\lVert g\rVert_h.
	\end{equation}
Take $\bm{v}_h = \bm{u}_h$ in (\ref{pb4}), and we obtain
	\begin{align*}
	\Pr\interleave \bm{u}_h\interleave^2&= d_h(F(\bm{u}_h),\bm{u}_h) + (\bm{f},\bm{u}_{h0})  \\
	&\leq (\Pr Ra\interleave F(\bm{u}_{h})\interleave + \lVert \bm{f}\rVert_h)\interleave \bm{u}_h\interleave \\
	&\leq (\Pr Ra\kappa ^{-1}\lVert g\rVert_h + \lVert \bm{f}\rVert_h)\interleave \bm{u}_h\interleave.
	\end{align*}
This indicates 
\begin{align}\label{stabs}
		 \interleave \bm{u}_h\interleave\leq   Ra\kappa ^{-1}\lVert g\rVert_h +{\Pr}^{-1} \lVert \bm{f}\rVert_h.
\end{align}		 
By Lemma \ref{boundedness} and (\ref{f_h}),  we also have
	\begin{equation*}
	\lvert -c_h(\bm{u}_h;\bm{u}_h,\bm{v}_h) +d_h(F(\bm{u}_h),\bm{v}_h) + (\bm{f},\bm{v}_{h0})\rvert
	\apprle \big( \interleave \bm{u}_h\interleave^2+\Pr  Ra\interleave F(\bm{u}_h)\interleave + \lVert \bm{f}\rVert_h\big)\interleave \bm{v}_h\interleave.
	\end{equation*}
	
	Now we  introduce another    mapping, $\mathcal{A}: \bm{W}_h\rightarrow \bm{W}_h$,  defined by $\mathcal{A}(\bm{u}_h) = \bm{w}$, where $\bm{w}\in  \bm{W}_h$ is   determined by
	\begin{equation}\label{pbw}
	a_h(\bm{w},\bm{v}_h) = -c_h(\bm{u}_h;\bm{u}_h,\bm{v}_h) + d_h(F(\bm{u}_h),\bm{v}_h) + (\bm{f},\bm{v}_{h0}), \forall \bm{v}_h\in \bm{W}_h.
	\end{equation}
		Clearly, $\bm{u}_h$ is a solution to (\ref{pb4}) if it is a solution to
	\begin{equation*}
	\mathcal{A}(\bm{u}_h) = \bm{u}_h.
	\end{equation*}
	To show this system has a solution, from the Leray-Schauder's principle it suffices to prove the following two assertions:
	 (i) $\mathcal{A}$ is a continuous and compact mapping;	 
	  (ii) for any   $0\leq \lambda \leq 1$, the set $\bm{W}_{\lambda,h}:=\{\bm{v}_h\in \bm{W}_h: \bm{v}_h = \lambda \mathcal{A}\bm{v}_h\}$ is bounded.
	  
	Let $\bm{u}_{1h},\bm{u}_{2h}\in \bm{W}_h$, set $\bm{w}_2 = \mathcal{A}(\bm{u}_{2h})$ and $\bm{w}_1 = \mathcal{A}(\bm{u}_{1h})$, then we obtain
	\begin{align}
	a_h(\bm{w}_1,\bm{v}_h) = -c_h(\bm{u}_{1h};\bm{u}_{1h},\bm{v}_h) + d_h(F(\bm{u}_{1h}),\bm{v}_h) + (\bm{f},\bm{v}_{h0}),\label{pbw1}\\
	a_h(\bm{w}_2,\bm{v}_h) = -c_h(\bm{u}_{2h};\bm{u}_{2h},\bm{v}_h) + d_h(F(\bm{u}_{2h}),\bm{v}_h) + (\bm{f},\bm{v}_{h0}).\label{pbw2}
	\end{align}  
	Subtracting \eqref{pbw2} from \eqref{pbw1},  and taking $\bm{v}_h =\bm{w}: = \bm{w}_1-\bm{w}_2$, we get
	\begin{equation}\label{355}
	a_h(\bm{w},\bm{w}) = -c_h(\bm{u}_{1h}-\bm{u}_{2h};\bm{u}_{1h},\bm{w}) -c_h(\bm{u}_{2h};\bm{u}_{1h}-\bm{u}_{2h},\bm{w}) + d_h(F(\bm{u}_{1h})-F(\bm{u}_{2h}),\bm{w}).
	\end{equation}
	Substitute $T_h=F(\bm{u}_{1h})$ and  $T_h=F(\bm{u}_{2h})$  into the second equation of (\ref{pb3}),  respectively, and subtract the two resultant equations each other, then, in view of \eqref{224}, we   have	  
	\begin{equation*}
	\overline{a}_h(F(\bm{u}_{1h})-F(\bm{u}_{2h}),s_h) = -\overline{c}_h(\bm{u}_{1h}-\bm{u}_{2h};F(\bm{u}_{1h}),s_h) -\overline{c}_h(\bm{u}_{2h};F(\bm{u}_{1h})-F(\bm{u}_{2h}),s_h) , \forall s_h\in S_h^0,
	\end{equation*}
	Taking $s_h = F(\bm{u}_{1h})-F(\bm{u}_{2h})$ in this equation, together with \eqref{3356}, \eqref{regular estimates Th}, and Lemma \ref{boundedness}, leads to 
	\begin{equation}\label{T1-T2}
	\begin{aligned}
	\kappa \interleave F(\bm{u}_{1h})-F(\bm{u}_{2h})\interleave &\apprle \interleave \bm{u}_{1h}-\bm{u}_{2h}\interleave\cdot \interleave F(\bm{u}_{1h})\interleave\\
	&\leq \kappa^{-1} \interleave \bm{u}_{1h}-\bm{u}_{2h}\interleave\cdot \lVert g\rVert_h.
	\end{aligned}
	\end{equation}
	As a result, from \eqref{355} and \eqref{stabs}  it follows
	\begin{align*}
	\interleave \mathcal{A}(\bm{u}_{1h})-\mathcal{A}(\bm{u}_{2h})\interleave&=\interleave \bm{w}\interleave\apprle ({\Pr}^{-1}(\interleave \bm{u}_{1h}\interleave+\interleave \bm{u}_{2h}\interleave) + \kappa^{-2}Ra\lVert g\rVert_h)\interleave \bm{u}_{1h}-\bm{u}_{2h}\interleave\\
	&\leq \big(2{\Pr}^{-1}(Ra\kappa ^{-1}\lVert g\rVert_h +{\Pr}^{-1} \lVert \bm{f}\rVert_h) + \kappa^{-2}Ra\lVert g\rVert_h\big)\interleave \bm{u}_{1h}-\bm{u}_{2h}\interleave,
	\end{align*}	
which means that   $\mathcal{A}$ is equicontinuous and uniformly bounded. Thus, $\mathcal{A}$ is compact by the Arzel\'{a}-Ascoli theorem\cite{brezis2010functional}.
	
	It remains to show (ii).  If $\lambda = 0$, then $\bm{W}_{\lambda,h}=\{0\}$.   For $\lambda \in (0,1]$ and  $\bm{v}_h\in \bm{W}_{\lambda,h}$, by  \eqref{pbw} and \eqref{3356}   we have
	\begin{align*}
	\lambda^{-1}a_h(\bm{v}_h,\bm{v}_h) =a_h(\mathcal{A}\bm{v}_h,\bm{v}_h) &= -c_h(\bm{v}_h;\bm{v}_h,\bm{v}_h) + d_h(F(\bm{v}_h),\bm{v}_h) + (\bm{f},\bm{v}_{h0})\\
&=  d_h(F(\bm{v}_h),\bm{v}_h) + (\bm{f},\bm{v}_{h0}),
	\end{align*}
which implies
	\begin{equation*}
	\interleave \bm{v}_h\interleave \leq \lambda Ra\interleave F(\bm{v}_h)\interleave + \lambda {\Pr}^{-1}\lVert \bm{f}\rVert_h\\
	\leq \lambda Ra\kappa^{-1}\lVert g\rVert_h + \lambda {\Pr}^{-1}\lVert \bm{f}\rVert_h.
	\end{equation*}
	This completes the proof.  
	\end{proof}

We now give a global uniqueness criteria for the case of small data (small Rayleigh number $Ra$).
\begin{thm}\label{thm32}
	 Suppose
	\begin{equation}\label{assu}
	({\Pr}^{-1}\mathcal{N}_hRa\kappa^{-1}+\mathcal{M}_h Ra\kappa^{-2})\lVert g\rVert_h + \mathcal{N}_h{\Pr}^{-2}\lVert \bm{f}\rVert_h < 1 .
	\end{equation}
	 Then the problem (\ref{pb3}) admits a unique solution $(\bm{u}_h,T_h)\in \bm{W}_h \times S_h^0$ with $T_h=F(\bm{u}_h)$. 
\end{thm}
\begin{proof} By Theorem \ref{exist}, let 
%
%
 $\bm{u}_{1h}, \bm{u}_{2h}\in \bm{W}_h$ be two  solutions to the problem (\ref{pb4}). Then  it suffices to show $\bm{u}_{1h}=\bm{u}_{2h}$.  In fact, we have
	\begin{align*}
	a_h(\bm{u}_{1h},\bm{v}_h) = -c_h(\bm{u}_{1h};\bm{u}_{1h},\bm{v}_h) + d_h(F(\bm{u}_{1h}),\bm{v}_h) + (\bm{f},\bm{v}_{h0}),\\
	a_h(\bm{u}_{2h},\bm{v}_h) = -c_h(\bm{u}_{2h};\bm{u}_{2h},\bm{v}_h) + d_h(F(\bm{u}_{2h}),\bm{v}_h) + (\bm{f},\bm{v}_{h0}).
	\end{align*}  
	Subtracting the above two equations each other with $\bm{v}_h = \bm{u}_{1h} - \bm{u}_{2h}$, and using \eqref{3356},  we obtain
	\begin{equation*}
	a_h(\bm{u}_{1h}-\bm{u}_{2h},\bm{u}_{1h}-\bm{u}_{2h}) = -c_h(\bm{u}_{1h}-\bm{u}_{2h};\bm{u}_{1h},\bm{u}_{1h}-\bm{u}_{2h}) + d_h(F(\bm{u}_{1h})-F(\bm{u}_{2h}),\bm{u}_{1h}-\bm{u}_{2h}),
	\end{equation*} 
which, together with Lemma \ref{boundedness}, (\ref{T1-T2}) and \eqref{stabs},  yields
	\begin{align*}
	\Pr  \interleave \bm{u}_{1h}-\bm{u}_{2h}\interleave^2 &\leq \mathcal{N}_h\interleave \bm{u}_{1h}-\bm{u}_{2h}\interleave^2\cdot \interleave \bm{u}_{1h}\interleave + \Pr  Ra\interleave F(\bm{u}_{1h})-F(\bm{u}_{2h})\interleave\cdot \interleave \bm{u}_{1h}-\bm{u}_{2h}\interleave\\
	&\leq \mathcal{N}_h\interleave \bm{u}_{1h}-\bm{u}_{2h}\interleave^2\cdot \interleave \bm{u}_{1h}\interleave + \mathcal{M}_h\Pr  Ra\kappa^{-2}\cdot \interleave \bm{u}_{1h}-\bm{u}_{2h}\interleave^2\cdot\lVert g\rVert_h,\\
	&\leq 
	\big((\mathcal{N}_hRa\kappa^{-1}+\mathcal{M}_h\Pr  Ra\kappa^{-2})\lVert g\rVert_h + \mathcal{N}_h{\Pr}^{-1}\lVert \bm{f}\rVert_h\big)\interleave \bm{u}_{1h}-\bm{u}_{2h}\interleave^2.
	\end{align*}
If $\bm{u}_{1h} \neq \bm{u}_{2h}$, then, by  the assumption \eqref{assu}, we further have
	\begin{align*}
	\Pr  \interleave \bm{u}_{1h}-\bm{u}_{2h}\interleave^2 
	< \Pr  \interleave \bm{u}_{1h}-\bm{u}_{2h}\interleave^2,
	\end{align*}
which contradicts.   Therefore $\bm{u}_{1h} = \bm{u}_{2h}$.
\end{proof}

%

\section{A priori error estimates}
This section is devoted to the error estimation of the WG scheme (\ref{pb2}). We set
\begin{equation*}
\bm{I}_h\bm{u}:=\{\bm{P}_k^{RT}\bm{u},\bm{Q}_l^b\bm{u}\},J_hp:=\{Q_{k-1}^0p,Q_k^bp\},H_h T:=\{Q_k^0T,Q_l^bT\}.
\end{equation*}
We recall that  $k\geq 1$ and $l=k, k-1$.
\begin{lem}\label{E-N}
	For any $\bm{w},\bm{u}\in \bm{W}$, $T\in H_0^1(\Omega)$, $\bm{v}_h\in \bm{V}_h^0$ and $s_h\in S_h^0$, it holds
	\begin{align}
	c_h(\bm{I}_h\bm{w};\bm{I}_h\bm{u},\bm{v}_h) =& (\nabla\cdot (\bm{u}\otimes\bm{w}),\bm{v}_{h0})+E_N(\bm{w};\bm{u},\bm{v}_h),\label{c_h}\\
	\overline{c}_h(\bm{I}_h\bm{u};H_hT,s_h) =& (\nabla\cdot (\bm{u}T),s_{h0})+\overline{E}_N(\bm{u};T,s_h),\label{c_hhat}
	\end{align}
	where 
	\begin{equation*}
	\begin{aligned}
	E_N(\bm{w};\bm{u},\bm{v}_h) :=& \frac{1}{2}(\bm{u}\otimes\bm{w}-\bm{P}_k^{RT}\bm{u}\otimes\bm{P}_k^{RT}\bm{w},\nabla_h\bm{v}_{h0})-\frac{1}{2}\langle (\bm{u}\otimes\bm{w}-\bm{Q}_l^b\bm{u}\otimes\bm{Q}_l^b\bm{w})\cdot\bm{n},\bm{v}_{h0}\rangle_{\partial \mathcal{T}_h^f}\\
	&\quad -\frac{1}{2}(\bm{w}\cdot\nabla\bm{u}-\bm{P}_k^{RT}\bm{w}\cdot\nabla_h\bm{P}_k^{RT}\bm{u},\bm{v}_{h0})-\frac{1}{2}\langle \bm{v}_{hb}\otimes\bm{Q}_l^b\bm{w}\cdot\bm{n},\bm{P}_k^{RT}\bm{u}\rangle_{\partial\mathcal{T}_h^f},\\
	\overline{E}_N(\bm{u};T,s_h) :=& \frac{1}{2}(\bm{u}T-\bm{P}_k^{RT}\bm{u}Q_k^0T,\nabla_h s_{h0})-\frac{1}{2}\langle (\bm{u}T-\bm{Q}_l^b\bm{u}Q_l^b T)\cdot\bm{n},s_{h0}\rangle_{\partial \mathcal{T}_h}\\
	&\quad -\frac{1}{2}(\bm{u}\cdot\nabla T-\bm{P}_K^{RT}\bm{u}\cdot\nabla_h Q_k^0T,s_{h0})-\frac{1}{2}\langle (\bm{Q}_l^b\bm{u}s_{hb})\cdot\bm{n},Q_k^0T\rangle_{\partial\mathcal{T}_h}.
	\end{aligned}
	\end{equation*}
\end{lem}
\begin{proof}
	From the definition of weak divergence and Green's formula, we have 
	\begin{equation*}
	\begin{aligned}
	(\nabla_{w,k}\cdot(\bm{I}_h\bm{u}\otimes\bm{I}_h\bm{w}),\bm{v}_{h0}) =& (\nabla\cdot (\bm{u}\otimes\bm{w}),\bm{v}_{h0})+
	(\bm{u}\otimes\bm{w}-\bm{P}_k^{RT}\bm{u}\otimes\bm{P}_k^{RT}\bm{w},\nabla_h\bm{v}_{h0})\\
	&-\langle (\bm{u}\otimes\bm{w}-\bm{Q}_l^b\bm{u}\otimes\bm{Q}_l^b\bm{w})\cdot\bm{n},\bm{v}_{h0}\rangle_{\partial \mathcal{T}_h^f},\\
	(\nabla_{w,k}\cdot(\bm{v}_h\otimes\bm{I}_h\bm{w}),\bm{P}_k^{RT}\bm{u}) =& (\nabla\cdot (\bm{u}\otimes\bm{w}),\bm{v}_{h0})+(\bm{w}\cdot\nabla\bm{u}-\bm{P}_K^{RT}\bm{w}\cdot\nabla_h\bm{P}_k^{RT}\bm{u},\bm{v}_{h0})\\
	&+\langle \bm{v}_{hb}\otimes\bm{Q}_l^b\bm{w}\cdot\bm{n},\bm{P}_k^{RT}\bm{u}\rangle_{\partial\mathcal{T}_h^f},
	\end{aligned}
	\end{equation*}
which, together with the definition of the trilinear form $c_h(\cdot;\cdot,\cdot) $, yields	 (\ref{c_h}). 

Similarly, we can obtain   (\ref{c_hhat}).
\end{proof}

\begin{lem}\label{RTgeneral}
	Let $j, r$ be  nonnegative integers. For any   $K\in \mathcal{T}_h$ and  $ \bm{v}\in [H^r(K)]^d$,  the following estimates hold for the RT projection operator:
	\begin{align}
	\lvert \bm{v}-\bm{P}_j^{RT}\bm{v}\rvert_{1,2,K}&\apprle h_K^{r-1}\lvert \bm{v}\rvert_{r,2,K},\forall 1\leq r\leq j+1,\label{RT12}  \\
	\lvert \bm{v}-\bm{P}_j^{RT}\bm{v}\rvert_{0,3,K}&\apprle h_K^{r-\frac{d}{6}}\lvert \bm{v}\rvert_{r,2,K},\forall 0\leq r\leq j+1,\label{RT03}\\
	\lvert \bm{v}-\bm{P}_j^{RT}\bm{v}\rvert_{0,2,\partial K}&\apprle h_K^{r-\frac{1}{2}}\lvert \bm{v}\rvert_{r,2,K},\forall 1\leq r\leq j+1,  \label{RT02edge} \\
	\lvert \bm{v}-\bm{P}_j^{RT}\bm{v}\rvert_{0,3,\partial K}&\apprle h_K^{r-\frac{1}{3}-\frac{d}{6}}\lvert \bm{v}\rvert_{r,2,K},\forall 1\leq r\leq j+1. \label{RT03edge}
	\end{align}
\end{lem}
\begin{proof}
	We only prove (\ref{RT12}), since the estimates (\ref{RT03})-(\ref{RT03edge}) follow similarly. 

For $ 1\leq r\leq j+1$, by the triangle inequality, the inverse inequality,   Lemma \ref{ineq}, and Lemma \ref{RT},  we get
	\begin{align*}
	\lvert \bm{v}-\bm{P}_j^{RT}\bm{v}\rvert_{1,2,K} &\leq \lvert \bm{v}-\bm{Q}_{r-1}^0\bm{v}\rvert_{1,2,K} +\lvert \bm{Q}_{r-1}^0\bm{v}-\bm{P}_j^{RT}\bm{v}\rvert_{1,2,K} \\
	&\apprle \lvert \bm{v}-\bm{Q}_{r-1}^0\bm{v}\rvert_{1,2,K} +h_K^{-1}\lvert \bm{Q}_{r-1}^0\bm{v}-\bm{P}_j^{RT}\bm{v}\rvert_{0,2,K} \\
	&\leq \lvert \bm{v}-\bm{Q}_{r-1}^0\bm{v}\rvert_{1,2,K} +h_K^{-1}\lvert \bm{Q}_{r-1}^0\bm{v}-\bm{v}\rvert_{0,2,K}+h_K^{-1}\lvert \bm{v}-\bm{P}_j^{RT}\bm{v}\rvert_{0,2,K} \\
	&\apprle  h_K^{r-1}\lvert \bm{v}\rvert_{r,2,K},
	\end{align*}
	i.e.   (\ref{RT12}) holds. 
\end{proof}
\begin{lem}\label{lem4.3}
	For $\bm{u}\in [H^{k+1}(\Omega_f)]^d$ with $\nabla\cdot\bm{u} = 0$ and $T\in H^{k+1}(\Omega)$, it holds
	\begin{align*}
	E_N(\bm{u};\bm{u},\bm{v}_h)&\apprle h^{k}\lVert \bm{u}\rVert_2\lVert \bm{u}\rVert_{k+1}\interleave \bm{v}_h\interleave, \forall \bm{v}_h\in \bm{V}_h^0,\\
	\overline{E}_N(\bm{u};T,s_h)&\apprle h^{k}(\lVert \bm{u}\rVert_2\lVert T\rVert_{k+1}+\lVert T\rVert_2\lVert \bm{u}\rVert_{k+1})\interleave s_h\interleave, \forall s_h\in S_h^0,
	\end{align*}
	for $l=k$ when $d = 2,3$, and for $l = k-1$ when $d=2$.
\end{lem}
\begin{proof}
	From the h\"{o}lder inequality, the sobolev inequality, and the projection properties, we have
	\begin{align*}
	&\lvert (\bm{u}\otimes \bm{u}-\bm{P}_k^{RT}\bm{u}\otimes\bm{P}_k^{RT}\bm{u},\nabla_h\bm{v}_{h0})\rvert \\
	\leq& \lvert ((\bm{u}-\bm{P}_k^{RT}\bm{u})\otimes \bm{u},\nabla_h\bm{v}_{h0})\rvert + \lvert (\bm{P}_k^{RT}\bm{u}\otimes(\bm{u}-\bm{P}_k^{RT}\bm{u}),\nabla_h\bm{v}_{h0})\rvert \\
	\apprle& \sum_{K\in \mathcal{T}_h}\lvert \bm{u}-\bm{P}_k^{RT}\bm{u}\rvert_{0,2,K}\lvert \bm{u}\rvert_{0,\infty,K}\lVert \nabla_h\bm{v}_{h0}\rVert_{0,2,K} + \sum_{K\in \mathcal{T}_h}\lvert \bm{u}-\bm{P}_k^{RT}\bm{u}\rvert_{0,2,K}\lvert \bm{P}_k^{RT}\bm{u}\rvert_{0,\infty,K}\lVert \nabla_h\bm{v}_{h0}\rVert_{0,2,K} \\
	\apprle& h^{k+1}\lvert \bm{u}\rvert_{0,\infty}\lvert \bm{u}\rvert_{k+1}\interleave \bm{v}_h\interleave \\
    \apprle& h^{k+1}\lVert \bm{u}\rVert_2\lVert \bm{u}\rVert_{k+1}\interleave \bm{v}_h\interleave.
	\end{align*}
	For $l=k$ when $d = 2,3$, and for $l = k-1$ when $d=2$, we have
	\begin{align*}
	&\lvert \langle (\bm{u}\otimes \bm{u}-\bm{Q}_l^b\bm{u}\otimes\bm{Q}_l^b\bm{u})\cdot \bm{n},\bm{v}_{h0}\rangle_{\partial\mathcal{T}_h^f}\rvert= \lvert \langle (\bm{u}\otimes \bm{u}-\bm{Q}_l^b\bm{u}\otimes\bm{Q}_l^b\bm{u})\cdot \bm{n},\bm{v}_{h0}-\bm{v}_{hb}\rangle_{\partial\mathcal{T}_h^f}\rvert \\
	\leq& \lvert \langle (\bm{u}-\bm{Q}_l^b\bm{u})\otimes \bm{u}\cdot \bm{n},\bm{v}_{h0}-\bm{v}_{hb}\rangle_{\partial\mathcal{T}_h^f}\rvert + \lvert \langle  \bm{Q}_l^b\bm{u}\otimes(\bm{u}-\bm{Q}_l^b\bm{u})\cdot \bm{n},\bm{v}_{h0}-\bm{v}_{hb}\rangle_{\partial\mathcal{T}_h^f}\rvert \\
	\apprle& \lvert \langle (\bm{u}-\bm{Q}_l^b\bm{u})\otimes (\bm{u}-\bm{Q}_k^0\bm{u})\cdot \bm{n},\bm{v}_{h0}-\bm{v}_{hb}\rangle_{\partial\mathcal{T}_h^f}\rvert + \lvert \langle (\bm{u}-\bm{Q}_l^b\bm{u})\otimes \bm{Q}_k^0\bm{u}\cdot \bm{n},\bm{v}_{h0}-\bm{v}_{hb}\rangle_{\partial\mathcal{T}_h^f}\rvert \\
	&+ \lvert \langle (\bm{Q}_l^0\bm{u}-\bm{Q}_l^b\bm{u})\otimes (\bm{u}-\bm{Q}_l^b\bm{u})\cdot \bm{n},\bm{v}_{h0}-\bm{v}_{hb}\rangle_{\partial\mathcal{T}_h^f}\rvert + \lvert \langle \bm{Q}_l^0\bm{u}\otimes(\bm{u}-\bm{Q}_l^b\bm{u}) \cdot \bm{n},\bm{v}_{h0}-\bm{v}_{hb}\rangle_{\partial\mathcal{T}_h^f}\rvert \\
	\leq& \sum_{K\in\mathcal{T}_h}(\lvert \bm{u}-\bm{Q}_l^b\bm{u}\rvert_{0,2,\partial K}\lvert \bm{u}-\bm{Q}_k^0\bm{u}\rvert_{0,2,\partial K}\lvert \bm{v}_{h0}-\bm{v}_{hb}\rvert_{0,\infty,\partial K} +\lvert \bm{u}-\bm{Q}_l^b\bm{u}\rvert_{0,2,\partial K}\lvert \bm{Q}_k^0\bm{u}\rvert_{0,\infty,\partial K}\lvert \bm{v}_{h0}-\bm{v}_{hb}\rvert_{0,2,\partial K}) \\
	&+\sum_{K\in\mathcal{T}_h}(\lvert \bm{Q}_l^0\bm{u}-\bm{Q}_l^b\bm{u}\rvert_{0,2,\partial K}\lvert \bm{u}-\bm{Q}_l^b\bm{u}\rvert_{0,2,\partial K}\lvert \bm{v}_{h0}-\bm{v}_{hb}\rvert_{0,\infty,\partial K} +\lvert \bm{u}-\bm{Q}_l^b\bm{u}\rvert_{0,2,\partial K}\lvert \bm{Q}_k^0\bm{u}\rvert_{0,\infty,\partial K}\lvert \bm{v}_{h0}-\bm{v}_{hb}\rvert_{0,2,\partial K}) \\
	\apprle& h^{l+1/2}\lvert \bm{u}\rvert_{l+1}h^{1/2}\lvert \bm{u}\rvert_1h^{1-d/2}\interleave \bm{v}_h\interleave + h^{l+1/2}\lvert \bm{u}\rvert_{l+1}\lvert \bm{u}\rvert_{0,\infty}h^{1/2}\interleave \bm{v}_h\interleave \\
	\apprle& h^{k}\lVert \bm{u}\rVert_{2}\lVert \bm{u}\rVert_{k+1}\interleave \bm{v}_h\interleave,
	\end{align*}	
	\begin{align*}
	&\lvert (\bm{u}\cdot \nabla\bm{u}-\bm{P}_k^{RT}\bm{u}\cdot\nabla_h\bm{P}_k^{RT}\bm{u},\bm{v}_{h0})\rvert \\
	\leq& \lvert ((\bm{u}-\bm{P}_k^{RT}\bm{u})\cdot \nabla\bm{u},\bm{v}_{h0})\rvert + \lvert (\bm{P}_k^{RT}\bm{u}\cdot
	 (\nabla\bm{u}-\nabla_h\bm{P}_k^{RT}\bm{u}),\bm{v}_{h0})\rvert \\
	\apprle& \sum_{K\in \mathcal{T}_h}\lvert \bm{u}-\bm{P}_k^{RT}\bm{u}\rvert_{0,3,K}\lvert \nabla\bm{u}\rvert_{0,2,K}\lVert \bm{v}_{h0}\rVert_{0,6,K} + \sum_{K\in \mathcal{T}_h}\lvert \nabla\bm{u}-\nabla\bm{P}_k^{RT}\bm{u}\rvert_{0,2,K}\lvert \bm{P}_k^{RT}\bm{u}\rvert_{0,\infty,K}\lVert \bm{v}_{h0}\rVert_{0,2,K} \\
	\apprle& h^{k+1-d/6}\lvert \bm{u}\rvert_{k+1}\lvert \bm{u}\rvert_{1,2}\lvert \bm{v}_{h0}\rvert_{1,2} + \lvert \bm{u}\rvert_{0,\infty}h^k\lvert \bm{u}\rvert_{k+1}\lvert \bm{v}_{h0}\rvert_{0,2} \\
	\apprle& h^{k}\lVert \bm{u}\rVert_2\lVert \bm{u}\rVert_{k+1}\interleave \bm{v}_h\interleave,
	\end{align*}
and	
	\begin{align*}
	&\lvert \langle \bm{v}_{hb}\otimes \bm{Q}_l^b\bm{u}\cdot \bm{n},\bm{P}_k^{RT}\bm{u}\rangle_{\partial\mathcal{T}_h^f}\rvert = \lvert \langle \bm{v}_{hb}\otimes \bm{Q}_l^b\bm{u}\cdot \bm{n},\bm{P}_k^{RT}\bm{u}-\bm{Q}_k^b\bm{u}\rangle_{\partial\mathcal{T}_h^f}\rvert \\
	\leq& \lvert \langle (\bm{v}_{h0}-\bm{v}_{hb})\otimes (\bm{Q}_l^b\bm{u}-\bm{Q}_l^0\bm{u})\cdot \bm{n},\bm{P}_k^{RT}\bm{u}-\bm{Q}_k^b\bm{u}\rangle_{\partial\mathcal{T}_h^f}\rvert + \lvert \langle \bm{v}_{h0}\otimes (\bm{Q}_l^b\bm{u}-\bm{Q}_l^0\bm{u})\cdot \bm{n},\bm{P}_k^{RT}\bm{u}-\bm{Q}_k^b\bm{u}\rangle_{\partial\mathcal{T}_h^f}\rvert \\
	&+\lvert \langle (\bm{v}_{h0}-\bm{v}_{hb})\otimes \bm{Q}_l^0\bm{u}\cdot \bm{n},\bm{P}_k^{RT}\bm{u}-\bm{Q}_k^b\bm{u}\rangle_{\partial\mathcal{T}_h^f}\rvert + \lvert \langle \bm{v}_{h0}\otimes \bm{Q}_l^0\bm{u}\cdot \bm{n},\bm{P}_k^{RT}\bm{u}-\bm{Q}_k^b\bm{u}\rangle_{\partial\mathcal{T}_h^f}\rvert \\
	\leq& \sum_{K\in\mathcal{T}_h}\lvert \bm{v}_{h0}-\bm{v}_{hb}\rvert_{0,\infty,\partial K}\lvert \bm{Q}_l^b\bm{u}-\bm{Q}_l^0\bm{u}\rvert_{0,2,\partial K}\lvert \bm{P}_k^{RT}\bm{u}-\bm{Q}_k^b\bm{u}\rvert_{0,2,\partial K} + \sum_{K\in\mathcal{T}_h}\lvert \bm{v}_{h0}\rvert_{0,\infty,\partial K}\lvert \bm{Q}_l^b\bm{u}-\bm{Q}_l^0\bm{u}\rvert_{0,2,\partial K}\lvert \bm{P}_k^{RT}\bm{u}-\bm{Q}_k^b\bm{u}\rvert_{0,2,\partial K} \\
	&+\sum_{K\in\mathcal{T}_h}\lvert \bm{v}_{h0}-\bm{v}_{hb}\rvert_{0,2,\partial K}\lvert \bm{Q}_l^0\bm{u}\rvert_{0,6,\partial K}\lvert \bm{P}_k^{RT}\bm{u}-\bm{Q}_k^b\bm{u}\rvert_{0,3,\partial K} + \sum_{K\in\mathcal{T}_h}\lvert \bm{v}_{h0}\rvert_{0,3,\partial K}\lvert \bm{Q}_l^0\bm{u}\rvert_{0,6,\partial K}\lvert \bm{P}_k^{RT}\bm{u}-\bm{Q}_k^b\bm{u}\rvert_{0,2,\partial K} \\
	\apprle& h^{1-d/2}\interleave \bm{v}_h\interleave h^{1/2}\lvert \bm{u}\rvert_1h^{k+1/2}\lvert \bm{u}\rvert_{k+1} + h^{-d/6}\lvert \bm{v}_{h0}\rvert_{0,6}h^{1/2}\lvert \bm{u}\rvert_1h^{k+1/2}\lvert \bm{u}\rvert_{k+1} \\
	&+ h^{1/2}\interleave \bm{v}_h\interleave h^{-1/6}\lvert \bm{u}\rvert_{0,6}h^{k+1-1/3-d/6}\lvert \bm{u}\rvert_{k+1} + h^{-1/3}\lvert \bm{v}_{h0}\rvert_{0,3}h^{-1/6}\lvert \bm{u}\rvert_{0,6}h^{k+1/2}\lvert \bm{u}\rvert_{k+1}  \\
	\apprle& h^{k}\lVert \bm{u}\rVert_{2}\lVert \bm{u}\rVert_{k+1}\interleave \bm{v}_h\interleave.
	\end{align*}
Similarly, we can obtain
\begin{align*}
&\lvert (\bm{u}T-\bm{P}_k^{RT}\bm{u}Q_k^0T,\nabla_hs_{h0})\rvert \\
\leq& \lvert (\bm{u}(T-Q_k^0T),\nabla_hs_{h0})\rvert + \lvert ((\bm{u}-\bm{P}_k^{RT}\bm{u})Q_k^0T,\nabla_hs_{h0})\rvert \\
\apprle& \sum_{K\in \mathcal{T}_h}\lvert T-Q_k^0T\rvert_{0,2,K}\lvert \bm{u}\rvert_{0,\infty,K}\lVert \nabla_hs_{h0}\rVert_{0,2,K} + \sum_{K\in \mathcal{T}_h}\lvert \bm{u}-\bm{P}_k^{RT}\bm{u}\rvert_{0,2,K}\lvert Q_k^0T\rvert_{0,\infty,K}\lVert \nabla_hs_{h0}\rVert_{0,2,K} \\
\apprle& h^{k+1}\lvert \bm{u}\rvert_{0,\infty}\lvert T\rvert_{k+1}\interleave s_h\interleave +h^{k+1}\lvert T\rvert_{0,\infty}\lvert \bm{u}\rvert_{k+1}\interleave s_h\interleave \\
\apprle& h^{k+1}\lVert \bm{u}\rVert_2\lVert T\rVert_{k+1}\interleave s_h\interleave +h^{k+1}\lVert T\rVert_2\lVert \bm{u}\rVert_{k+1}\interleave s_h\interleave,
\end{align*}
\begin{align*}
&\lvert (\bm{u}\cdot \nabla T-\bm{P}_k^{RT}\bm{u}\cdot\nabla_h Q_k^0T,s_{h0})\rvert \\
\leq& \lvert ((\bm{u}-\bm{P}_k^{RT}\bm{u})\cdot \nabla T,s_{h0})\rvert + \lvert (\bm{P}_k^{RT}\bm{u}\cdot
(\nabla T-\nabla_h Q_k^0T),s_{h0})\rvert \\
\apprle& \sum_{K\in \mathcal{T}_h}\lvert \bm{u}-\bm{P}_k^{RT}\bm{u}\rvert_{0,3,K}\lvert \nabla T\rvert_{0,2,K}\lVert s_{h0}\rVert_{0,6,K} + \sum_{K\in \mathcal{T}_h}\lvert \nabla T-\nabla_h Q_k^0T\rvert_{0,2,K}\lvert \bm{P}_k^{RT}\bm{u}\rvert_{0,\infty,K}\lVert s_{h0}\rVert_{0,2,K} \\
\apprle& h^{k+1-d/6}\lvert \bm{u}\rvert_{k+1}\lvert T\rvert_{1,2}\lvert s_{h0}\rvert_{1,2} + \lvert \bm{u}\rvert_{0,\infty}h^k\lvert T\rvert_{k+1}\lvert s_{h0}\rvert_{0,2} \\
\apprle& h^{k}\lVert \bm{u}\rVert_2\lVert T\rVert_{k+1}\interleave s_h\interleave +h^{k}\lVert T\rVert_2\lVert \bm{u}\rVert_{k+1}\interleave s_h\interleave.
\end{align*}
For $l=k$ when $d = 2,3$, and for $l = k-1$ when $d=2$, we have
\begin{align*}
&\lvert \langle (\bm{u}T-\bm{Q}_l^b\bm{u}Q_l^bT)\cdot \bm{n},s_{h0}\rangle_{\partial\mathcal{T}_h}\rvert =
 \lvert \langle (\bm{u}T-\bm{Q}_l^b\bm{u}Q_l^bT)\cdot \bm{n},s_{h0}-s_{hb}\rangle_{\partial\mathcal{T}_h}\rvert \\
\leq& \lvert \langle (\bm{u}(T-Q_l^bT))\cdot \bm{n},s_{h0}-s_{hb}\rangle_{\partial\mathcal{T}_h}\rvert + \lvert \langle ((\bm{u}-\bm{Q}_l^b\bm{u})Q_l^bT)\cdot \bm{n},s_{h0}-s_{hb}\rangle_{\partial\mathcal{T}_h}\rvert \\
\apprle& \lvert \langle (T-Q_l^bT)(\bm{u}-\bm{Q}_k^0\bm{u})\cdot \bm{n},s_{h0}-s_{hb}\rangle_{\partial\mathcal{T}_h}\rvert + \lvert \langle (T-Q_l^bT)\bm{Q}_k^0\bm{u}\cdot \bm{n},s_{h0}-s_{hb}\rangle_{\partial\mathcal{T}_h}\rvert \\
& +\lvert \langle (Q_l^bT-Q_l^0T)(\bm{u}-\bm{Q}_l^b\bm{u})\cdot \bm{n},s_{h0}-s_{hb}\rangle_{\partial\mathcal{T}_h}\rvert + \lvert \langle Q_l^0T(\bm{u}-\bm{Q}_l^b\bm{u})\cdot \bm{n},s_{h0}-s_{hb}\rangle_{\partial\mathcal{T}_h}\rvert \\
\leq& \sum_{K\in\mathcal{T}_h}(\lvert T-Q_l^bT\rvert_{0,2,\partial K}\lvert \bm{u}-\bm{Q}_k^0\bm{u}\rvert_{0,2,\partial K}\lvert s_{h0}-s_{hb}\rvert_{0,\infty,\partial K} +\lvert T-Q_l^bT\rvert_{0,2,\partial K}\lvert \bm{Q}_k^0\bm{u}\rvert_{0,\infty,\partial K}\lvert s_{h0}-s_{hb}\rvert_{0,2,\partial K}) \\
& +\sum_{K\in\mathcal{T}_h}(\lvert Q_l^bT-Q_l^0T\rvert_{0,2,\partial K}\lvert \bm{u}-\bm{Q}_l^b\bm{u}\rvert_{0,2,\partial K}\lvert s_{h0}-s_{hb}\rvert_{0,\infty,\partial K} +\lvert Q_l^0T\rvert_{0,\infty,\partial K}\lvert \bm{u}-\bm{Q}_l^b\bm{u}\rvert_{0,2,\partial K}\lvert s_{h0}-s_{hb}\rvert_{0,2,\partial K}) \\
\apprle& h^{l+1/2}\lvert T\rvert_{l+1}h^{1/2}\lvert \bm{u}\rvert_1h^{1-d/2}\interleave s_h\interleave + h^{l+1/2}\lvert T\rvert_{l+1}\lvert \bm{u}\rvert_{0,\infty}h^{1/2}\interleave s_h\interleave \\
&+h^{l+1/2}\lvert T\rvert_{l+1}h^{1/2}\lvert \bm{u}\rvert_1h^{1-d/2}\interleave s_h\interleave + \lvert T\rvert_{0,\infty}h^{l+1/2}\lvert \bm{u}\rvert_{l+1}h^{1/2}\interleave s_h\interleave \\
\apprle& h^{k}\lVert \bm{u}\rVert_{2}\lVert T\rVert_{k+1}\interleave s_h\interleave + h^{k}\lVert T\rVert_{2}\lVert \bm{u}\rVert_{k+1}\interleave s_h\interleave,
\end{align*}
\begin{align*}
&\lvert \langle s_{hb}\bm{Q}_l^b\bm{u}\cdot \bm{n},Q_k^0T\rangle_{\partial\mathcal{T}_h}\rvert= \lvert \langle s_{hb}\bm{Q}_l^b\bm{u}\cdot \bm{n},Q_k^0T-Q_k^bT\rangle_{\partial\mathcal{T}_h}\rvert \\
\leq& \lvert \langle (s_{h0}-s_{hb})(\bm{Q}_l^b\bm{u}-\bm{Q}_l^0\bm{u})\cdot \bm{n},Q_k^0T-Q_k^bT\rangle_{\partial\mathcal{T}_h}\rvert + \lvert \langle s_{h0}(\bm{Q}_l^b\bm{u}-\bm{Q}_l^0\bm{u})\cdot \bm{n},Q_k^0T-Q_k^bT\rangle_{\partial\mathcal{T}_h}\rvert \\
&+\lvert \langle (s_{h0}-s_{hb})\bm{Q}_l^0\bm{u}\cdot \bm{n},Q_k^0T-Q_k^bT\rangle_{\partial\mathcal{T}_h}\rvert + \lvert \langle s_{h0}\bm{Q}_l^0\bm{u}\cdot \bm{n},Q_k^0T-Q_k^bT\rangle_{\partial\mathcal{T}_h}\rvert \\
\leq& \sum_{K\in\mathcal{T}_h}\lvert s_{h0}-s_{hb}\rvert_{0,\infty,\partial K}\lvert \bm{Q}_l^b\bm{u}-\bm{Q}_l^0\bm{u}\rvert_{0,2,\partial K}\lvert Q_k^0T-Q_k^bT\rvert_{0,2,\partial K} + \sum_{K\in\mathcal{T}_h}\lvert s_{h0}\rvert_{0,\infty,\partial K}\lvert \bm{Q}_l^b\bm{u}-\bm{Q}_l^0\bm{u}\rvert_{0,2,\partial K}\lvert Q_k^0T-Q_k^bT\rvert_{0,2,\partial K} \\
&+\sum_{K\in\mathcal{T}_h}\lvert s_{h0}-s_{hb}\rvert_{0,2,\partial K}\lvert \bm{Q}_l^0\bm{u}\rvert_{0,6,\partial K}\lvert Q_k^0T-Q_k^bT\rvert_{0,3,\partial K} + \sum_{K\in\mathcal{T}_h}\lvert s_{h0}\rvert_{0,3,\partial K}\lvert \bm{Q}_l^0\bm{u}\rvert_{0,6,\partial K}\lvert Q_k^0T-Q_k^bT\rvert_{0,2,\partial K} \\
\apprle& h^{1-d/2}\interleave s_h\interleave h^{1/2}\lvert \bm{u}\rvert_1h^{k+1/2}\lvert T\rvert_{k+1} + h^{-d/6}\lvert s_{h0}\rvert_{0,6}h^{1/2}\lvert \bm{u}\rvert_1h^{k+1/2}\lvert T\rvert_{k+1} \\
&+ h^{1/2}\interleave s_h\interleave h^{-1/6}\lvert \bm{u}\rvert_{0,6}h^{k+2/3-d/6}\lvert T\rvert_{k+1} + h^{-1/3}\lvert s_{h0}\rvert_{0,3}h^{-1/6}\lvert \bm{u}\rvert_{0,6}h^{k+1/2}\lvert T\rvert_{k+1}  \\
\apprle& h^k\lVert \bm{u}\rVert_{2}\lVert T\rVert_{k+1}\interleave s_h\interleave.
\end{align*}
As a result, the two desired results follow from  the definitions of $E_N(\bm{u};\bm{u},\bm{v}_h)$,
	$\overline{E}_N(\bm{u};T,s_h) $ given in Lemma \ref{E-N}.
\end{proof}

\begin{lem}\label{lem44}
	Let $(\bm{u},p,T)$ be the solution to the problem (\ref{pb1}), then it holds
    \begin{eqnarray}\label{Iu}
	A_h(\bm{I}_h\bm{u};\bm{I}_h\bm{u},\bm{v}_h)&+&b_h(\bm{v}_h,J_hp)-b_h(\bm{u}_h,q_h)-d_h(H_hT,\bm{v}_h)\nonumber \\
	 &=& (\bm{f},\bm{v}_{h0})+E_L(\bm{u},\bm{v}_h)+E_N(\bm{u};\bm{u},\bm{v}_h), \forall (\bm{v}_h,q_h)\in \bm{V}_h^0\times Q_h^0, 
\label{HT}\\
	\overline{A}_h(\bm{I}_h\bm{u};H_hT,s_h)&=&(g,s_{h0})+\overline{E}_L(T,s_h)+\overline{E}_N(\bm{u};T,s_h), \forall s_h\in S_h^0,\label{411}
    \end{eqnarray}
    where 
    \begin{align*}
    E_L(\bm{u},\bm{v}_h) &:= \Pr  \langle(\nabla\bm{u}-\bm{Q}_m^0\nabla\bm{u})\cdot \bm{n},\bm{v}_{h0}-\bm{v}_{hb}\rangle_{\partial \mathcal{T}_h^f}+\Pr  \langle\tau(\bm{P}_k^{RT}\bm{u}-\bm{u}),\bm{v}_{h0}-\bm{v}_{hb}\rangle_{\partial \mathcal{T}_h^f},\\
    \overline{E}_L(T,s_h) &:= \kappa\langle(\nabla T-\bm{Q}_m^0\nabla T)\cdot \bm{n},s_{h0}-s_{hb}\rangle_{\partial \mathcal{T}_h}+\kappa\langle\tau(Q_k^0 T-T),s_{h0}-s_{hb}\rangle_{\partial \mathcal{T}_h}.
    \end{align*}
    In addition, it holds
    \begin{equation}\label{local RT}
    \bm{P}_k^{RT}\bm{u}|_K\in [P_k(K)]^d,\quad \forall K\in \mathcal{T}_h^f.
    \end{equation} 
\end{lem}
\begin{proof} We first show \eqref{local RT}. In fact,
	for all $K\in \mathcal{T}_h,\varphi \in P_k(K)$, by  Lemma \ref{RT38}  we get 
\begin{equation*}
	(\nabla\cdot \bm{P}_k^{RT}\bm{u},\varphi)_K= (\nabla\cdot \bm{u},\varphi)_K=0,
	\end{equation*}
	which indicates 
	\begin{equation}\label{div415}
	\nabla\cdot \bm{P}_k^{RT}\bm{u}=0.
	\end{equation}

	 Thus, the result (\ref{local RT}) follows from   Lemma \ref{RT3.6}.
	
	By the definition of $a_0(\cdot,\cdot),A(\cdot,\cdot)$ and $d_h(\cdot,\cdot)$, we obtain
	\begin{equation}\label{R}
	\begin{aligned}
	&A_h(\bm{I}_h\bm{u};\bm{I}_h\bm{u},\bm{v}_h)+b_h(\bm{v}_h,J_hp)-b_h(\bm{u}_h,q_h)-d_h(H_hT,\bm{v}_h)\\
	=&\Pr (\nabla_{w,m}\bm{I}_h\bm{u},\nabla_{w,m}\bm{v}_h)\\
	&+\Pr \langle \tau\bm{Q}_l^b(\bm{P}_k^{RT}\bm{u}-\bm{u}),\bm{Q}_l^b\bm{v}_{h0}-\bm{v}_{hb}\rangle_{\partial \mathcal{T}_h^f}\\
	&+\frac{1}{2}(\nabla_{w,k}\cdot (\bm{I}_h\bm{u}\otimes \bm{I}_h\bm{u}),\bm{v}_{h0})-\frac{1}{2}(\nabla_{w,k}\cdot (\bm{v}_h\otimes \bm{I}_h\bm{u}),\bm{P}_k^{RT}\bm{u})\\
	&+(\nabla_{w,k}\{Q_{k-1}^0p,Q_k^bp\},\bm{v}_{h0})\\
	&-(\nabla_{w,k}q_h,\bm{P}_k^{RT}\bm{u})\\
	&-\Pr  Ra(\bm{j}Q_k^0T,\bm{v}_{h0})\\
	:=&\sum_{i=1}^{6}R_i.
	\end{aligned}
	\end{equation}
	From the commutativity property (\ref{com1}), the definition of weak gradient,  Green's formula, the property of the projection $\bm{Q}_l^b$, and the relation  $\langle \nabla\bm{u}\cdot \bm{n},\bm{v}_{hb}\rangle_{\partial \mathcal{T}_h}=0$, it follows
	\begin{equation}\label{R1}
	\begin{aligned}
	R_1&=\Pr (\bm{Q}_m^0\nabla\bm{u},\nabla_{w,k}\bm{v}_h)\\
	&=-\Pr (\nabla_h\cdot \bm{Q}_m^0\nabla\bm{u},\bm{v}_{h0})+\Pr\langle \bm{Q}_m^0\nabla\bm{u}\cdot \bm{n},\bm{v}_{hb}\rangle_{\partial \mathcal{T}_h^f}\\
	&=\Pr (\bm{Q}_m^0\nabla\bm{u},\nabla_h\bm{v}_{h0})+\Pr\langle\bm{Q}_m^0\nabla\bm{u}\cdot \bm{n},\bm{v}_{hb}-\bm{v}_{h0}\rangle_{\partial \mathcal{T}_h^f}\\
	&=-\Pr (\Delta\bm{u},\bm{v}_{h0})+\Pr\langle (\nabla\bm{u}-\bm{Q}_m^0\nabla\bm{u})\cdot \bm{n},\bm{v}_{hb}-\bm{v}_{h0}\rangle_{\partial \mathcal{T}_h^f}.
	\end{aligned}
	\end{equation}
	By the definitions of the projections $\bm{Q}_l^b$ and $Q_k^0$, we have
	\begin{equation}\label{R2}
	R_2=\Pr \langle\tau(\bm{P}_k^{RT}\bm{u}-\bm{u}),\bm{Q}_l^b\bm{v}_{h0}-\bm{v}_{hb}\rangle_{\partial \mathcal{T}_h^f},\\
		\end{equation}
	\begin{equation}\label{R6}
	R_6 = \Pr  Ra(\bm{j}Q_k^0T,\bm{v}_{h0}) = \Pr  Ra(\bm{j}T,\bm{v}_{h0}).
	\end{equation}
By (\ref{c_h}), we get
	\begin{align}
	R_3 =& (\nabla\cdot (\bm{u}\otimes\bm{u}),\bm{v}_{h0})+E_N(\bm{u};\bm{u},\bm{v}_h).
	\end{align}
	The commutativity property (\ref{com2}) gives
	\begin{equation}\label{R4}
	R_4=(\bm{Q}_k^0\nabla p,\bm{v}_{h0})=(\nabla p,\bm{v}_{h0}).
	\end{equation}
	In view of  \eqref{div415}, (\ref{RT1}), and the definitions of $d_h(\cdot,\cdot)$ and weak gradient, we obtain
	\begin{equation}\label{R5}
	\begin{aligned}
	R_5&=-(\bm{P}_k^{RT}\bm{u},\nabla_{w,k}q_h)\\
	&=(\nabla\cdot \bm{P}_k^{RT}\bm{u},q_{h0})-\langle \bm{P}_k^{RT}\bm{u}\cdot \bm{n},q_{hb}\rangle_{\partial \mathcal{T}_h^f}\\
	&=-\langle \bm{u}\cdot \bm{n},q_{hb}\rangle_{\partial \mathcal{T}_h^f}\\
	&=0.
	\end{aligned}
	\end{equation}
	Finally, the desired relation \eqref{HT} follows from the combination of  (\ref{R})-(\ref{R5}) and the first equation of \eqref{pb1}.
	
Similarly, we can get  the relation (\ref{411}). This completes the proof.
\end{proof}
\begin{lem}\label{lem4.5}
	For $\bm{u}\in [H^{k+1}(\Omega_f)]^d$ and $T\in H^{k+1}(\Omega)$,   it holds
	\begin{align}
	\lvert E_L(\bm{u},\bm{v}_h)\rvert &\apprle \Pr  h^k\lvert \bm{u}\rvert_{k+1} \interleave \bm{v}_h \interleave,\quad \forall \bm{v}_h\in \bm{V}_h^0,\label{ELerroru}\\
	\lvert \overline{E}_L(T,s_h)\rvert &\apprle \kappa  h^k\lvert T\rvert_{k+1} \interleave s_h \interleave,\quad \forall s_h\in S_h^0.\label{ELerrorT}
	\end{align}
\end{lem}
\begin{proof}
	 By Lemma \ref{ineq} and the definition of $\bm{Q}_m^0$, we have
\begin{equation*}
	\begin{aligned}
	\lvert E_L(\bm{u},\bm{v}_h)\rvert&=\lvert \Pr\langle(\nabla\bm{u}-\bm{Q}_m^0\nabla\bm{u})\cdot \bm{n},\bm{v}_{h0}-\bm{v}_{hb}\rangle_{\partial \mathcal{T}_h^f}\rvert+\lvert\Pr\langle\tau(\bm{P}_k^{RT}\bm{u}-\bm{u}),\bm{Q}_l^b\bm{v}_{h0}-\bm{v}_{hb}\rangle_{\partial \mathcal{T}_h^f}\rvert\\
	&\leq \sum_{K\in\mathcal{T}_h^f}\Pr\lVert \nabla\bm{u}-\bm{Q}_m^0\nabla\bm{u}\rVert_{0,\partial K}(\lVert \bm{v}_{h0}-\bm{Q}_l^b\bm{v}_{h0}\rVert_{0,\partial K} +\lVert \bm{Q}_l^b\bm{v}_{h0}-\bm{v}_{hb}\rVert_{0,\partial K})\\
	&\quad  +\sum_{K\in\mathcal{T}_h^f}\Pr\lVert \tau^{1/2}(\bm{P}_k^{RT}\bm{u}-\bm{u})\rVert_{0,\partial K}\lVert \tau^{1/2}(\bm{Q}_l^b\bm{v}_{h0}-\bm{v}_{hb})\rVert_{0,\partial K}\\
	&\apprle \Pr h^k\lvert \bm{u}\rvert_{k+1}(\lVert \nabla_h\bm{v}_{h0}\rVert_0 +\lVert \tau^{1/2}(\bm{v}_{h0}-\bm{v}_{hb})\rVert_{0,\partial\mathcal{T}_h^f}) +\Pr h^k\lvert \bm{u}\rvert_{k+1}\lVert\lvert \bm{v}_h\rvert\rVert \\
	&\apprle \Pr h^k\lvert \bm{u}\rvert_{k+1}\lVert\lvert \bm{v}_h\rvert\rVert.
	\end{aligned}
	\end{equation*}	
i.e. \eqref{ELerroru} holds. 	The estimate (\ref{ELerrorT}) follows similarly. 
\end{proof}

\begin{thm}\label{error1}
	Let $(\bm{u},p,T)\in [H^{k+1}(\Omega_f)]^d\times H^k(\Omega_f)\times H^{k+1}(\Omega)$ and $(\bm{u}_h,p_h,T_h)\in \bm{V}_h^0\times Q_h^0\times S_h^0$ be the solutions to the problem (\ref{pb1}) and the WG scheme (\ref{pb2}), respectively. Then, under the assumption \eqref{assu} with 
	 \begin{equation}\label{c-0}
	 C_0:=1 -({\Pr}^{-1}\mathcal{N}_hRa\kappa^{-1}+\mathcal{M}_h Ra\kappa^{-2})\lVert g\rVert_h + \mathcal{N}_h{\Pr}^{-2}\lVert \bm{f}\rVert_h> 0 ,
	 \end{equation}
	it holds the following estimates: for $l=k$ when $d = 2,3$, and for $l = k-1$ when $d=2$,
	\begin{align}
	&\interleave \bm{I}_h\bm{u}-\bm{u}_h\interleave \apprle C_1h^k\lVert \bm{u}\rVert_{k+1},\label{erroru1}\\
    &\interleave H_h T-T_h\interleave \apprle (C_1\mathcal{M}_h\kappa^{-2}\lVert g\rVert_h +\kappa^{-1}\lVert T\rVert_2)h^k\lVert \bm{u}\rVert_{k+1}+(1 +\kappa^{-1}\lVert \bm{u}\rVert_2)h^k\lVert T\rVert_{k+1},\label{errorT1}\\
    &\lVert J_hp-p_h\rVert \apprle \Pr (C_1+ Ra\kappa^{-1}\lVert T\rVert_2)h^k\lVert \bm{u}\rVert_{k+1}+\Pr Ra(1 +\kappa^{-1}\lVert \bm{u}\rVert_2)h^k\lVert T\rVert_{k+1},\label{errorp}
    \end{align}
    where $C_1 := \frac{1+{\Pr}^{-1}\lVert \bm{u}\rVert_2}{C_0+\mathcal{M}_hRa\kappa^{-2}\lVert g\rVert_h}$.
\end{thm}
\begin{proof}
	From (\ref{pb2}) and Lemma \ref{lem44} we   easily get the following error equations:
	\begin{equation}\label{430}
	\begin{aligned}
	&a_h(\bm{I}_h\bm{u}-\bm{u}_h,\bm{v}_h)+c_h(\bm{I}_h\bm{u};\bm{I}_h\bm{u},\bm{v}_h)-c_h(\bm{u}_h;\bm{u}_h,\bm{v}_h)+b_h(\bm{v}_h,J_hp-p_h) \\
	&\quad -b_h(\bm{I}_h\bm{u}-\bm{u}_h,q_h)
	-d_h(H_hT-T_h,\bm{v}_h)=E_L(\bm{u},\bm{v}_h)+E_N(\bm{u};\bm{u},\bm{v}_h),\forall (\bm{v}_h,q_h)\in \bm{V}_h^0\times Q_h^0,\\
	\end{aligned}
	\end{equation}
	\begin{equation}
	\overline{a}_h(H_hT-T_h,s_h)+\overline{c}_h(\bm{I}_h\bm{u};H_hT,s_h)-\overline{c}_h(\bm{u}_h;T_h,s_h)=\overline{E}_L(T,s_h)+\overline{E}_N(\bm{u};T,s_h), \forall s_h\in S_h^0. 
	\end{equation}
	Take   $(\bm{v}_h,q_h,s_h) = (\bm{I}_h\bm{u}-\bm{u}_h,J_hp-p_h,H_hT-T_h)$ in the above two equations, then we have
	\begin{equation}\label{431}
	\begin{aligned}
	\Pr \interleave \bm{I}_h\bm{u}-\bm{u}_h\interleave^2 &= E_L(\bm{u},\bm{I}_h\bm{u}-\bm{u}_h)+E_N(\bm{u};\bm{u},\bm{I}_h\bm{u}-\bm{u}_h)\\
	&-c_h(\bm{I}_h\bm{u};\bm{I}_h\bm{u},\bm{I}_h\bm{u}-\bm{u}_h)+c_h(\bm{u}_h;\bm{u}_h,\bm{I}_h\bm{u}-\bm{u}_h),
	\end{aligned}
	\end{equation}
	\begin{equation}
	\begin{aligned}
	\kappa \interleave H_h T-T_h\interleave^2 &= \overline{E}_L(T,H_h T-T_h)+\overline{E}_N(\bm{u};T,H_h T-T_h)\\
	&-\overline{c}_h(\bm{I}_h\bm{u};H_hT,H_h T-T_h)+\overline{c}_h(\bm{u}_h;T_h,H_h T-T_h)
	\end{aligned}
	\end{equation}
		By \eqref{3356} and (\ref{stabs}), it holds
	\begin{equation*}
	\begin{aligned}
	&c_h(\bm{I}_h\bm{u};\bm{I}_h\bm{u},\bm{I}_h\bm{u}-\bm{u}_h)-c_h(\bm{u}_h;\bm{u}_h,\bm{I}_h\bm{u}-\bm{u}_h) \\
	=& c_h(\bm{I}_h\bm{u}-\bm{u}_h;\bm{u}_h,\bm{I}_h\bm{u}-\bm{u}_h)\\
\leq& \mathcal{N}_h\interleave \bm{u}_h\interleave \interleave \bm{I}_h\bm{u}-\bm{u}_h\interleave^2\\
	\leq& \mathcal{N}_h(Ra\kappa^{-1}\lVert g\rVert_h+{\Pr}^{-1}\lVert \bm{f}\rVert_h)\interleave \bm{I}_h\bm{u}-\bm{u}_h\interleave^2,
	\end{aligned}
	\end{equation*}
which, together with \eqref{431}, \eqref{c-0}, Lemma \ref{lem4.3}, and Lemma \ref{lem4.5},  leads to
	\begin{equation*}
    \begin{aligned}
     (C_0+\mathcal{M}_h Ra\kappa^{-2}\lVert g\rVert_h)\interleave \bm{I}_h\bm{u}-\bm{u}_h\interleave 
     \apprle (1+{\Pr}^{-1}\lVert \bm{u}\rVert_2)h^k\lVert \bm{u}\rVert_{k+1},
    \end{aligned}
	\end{equation*}
i.e.   \eqref{erroru1} holds.

	Similarly, we can obtain
	\begin{align*}
	 \interleave H_h T-T_h\interleave &\leq \mathcal{M}_h\kappa^{-2}\lVert g\rVert_h\interleave \bm{I}_h\bm{u}-\bm{u}_h\interleave + \kappa^{-1}(\kappa +\lVert \bm{u}\rVert_2)h^k\lVert T\rVert_{k+1} + \kappa^{-1}h^k\lVert T\rVert_2\lVert \bm{u}\rVert_{k+1}\\
	 & \apprle (C_1\mathcal{M}_h\kappa^{-2}\lVert g\rVert_h +\kappa^{-1}\lVert T\rVert_2)h^k\lVert \bm{u}\rVert_{k+1}+(1 +\kappa^{-1}\lVert \bm{u}\rVert_2)h^k\lVert T\rVert_{k+1},
	\end{align*}
	i.e. \eqref{errorT1} holds.
	
	Finally, let us estimate $\lVert J_hp-p_h\rVert$. In light of Theorem \ref{infsup}, \eqref{430}, Lemma \ref{boundedness},  Lemma \ref{lem4.3},  Lemma \ref{lem4.5},  \eqref{erroru1}, and  \eqref{errorT1}, we have
	\begin{align*}
	&\lVert J_hp-p_h\rVert \apprle \sup\limits_{\bm{0}\neq\bm{v}_h\in \bm{V}_h^0} \frac{b_h(\bm{v}_h,J_hp-p_h)}{\interleave \bm{v}_h\interleave}\\
	=& \sup\limits_{\bm{0}\neq\bm{v}_h\in \bm{V}_h^0} \frac{E_L(\bm{u},\bm{v}_h)+E_N(\bm{u};\bm{u},\bm{v}_h)-a_h(\bm{I}_h\bm{u}-\bm{u}_h,\bm{v}_h)-c_h(\bm{I}_h\bm{u};\bm{I}_h\bm{u},\bm{v}_h)+c_h(\bm{u}_h;\bm{u}_h,\bm{v}_h)+d_h(H_hT-T_h,\bm{v}_h)}{\interleave \bm{v}_h\interleave}\\
	\apprle& (\Pr +\lVert \bm{u}\rVert_2)h^k\lVert \bm{u}\rVert_{k+1}+(\Pr + \mathcal{N}_h(Ra\kappa^{-1}\lVert g\rVert_h+{\Pr}^{-1}\lVert \bm{f}\rVert_h))\interleave \bm{I}_h\bm{u}-\bm{u}_h\interleave + \Pr  Ra\interleave H_h T-T_h\interleave\\
	\apprle& \Pr (C_1+ Ra\kappa^{-1}\lVert T\rVert_2)h^k\lVert \bm{u}\rVert_{k+1}+\Pr Ra(1 +\kappa^{-1}\lVert \bm{u}\rVert_2)h^k\lVert T\rVert_{k+1},
	\end{align*}
	i.e. \eqref{errorp} holds.
\end{proof}

From Theorem \ref{error1}, Lemma \ref{lem3.3}, and the triangle inequality, it follows the following error estimates:
\begin{thm}
	Under the same conditions of  Theorem \ref{error1}, it holds
	\begin{align*}
	&\lVert \nabla\bm{u}-\nabla_h\bm{u}_{h0}\rVert_0 \apprle (C_1+1)h^k\lVert \bm{u}\rVert_{k+1},\\
	&\lVert \nabla\bm{u}-\nabla_{w,m}\bm{u}_h\rVert_0 \apprle (C_1+1)h^k\lVert \bm{u}\rVert_{k+1},\\
	&\lVert \nabla T-\nabla_h T_{h0}\rVert_0 \apprle (\mathcal{M}_h\kappa^{-2}\lVert g\rVert_hC_1 +\kappa^{-1}\lVert T\rVert_2)h^k\lVert \bm{u}\rVert_{k+1}+(1 +\kappa^{-1}\lVert \bm{u}\rVert_2)h^k\lVert T\rVert_{k+1},\\
	&\lVert \nabla T-\nabla_{w,m}T_h\rVert_0 \apprle (\mathcal{M}_h\kappa^{-2}\lVert g\rVert_hC_1 +\kappa^{-1}\lVert T\rVert_2)h^k\lVert \bm{u}\rVert_{k+1}+(1 +\kappa^{-1}\lVert \bm{u}\rVert_2)h^k\lVert T\rVert_{k+1},\\
	&\lVert p-p_{h0}\rVert_0 \apprle \Pr (C_1+ Ra\kappa^{-1}\lVert T\rVert_2)h^k\lVert \bm{u}\rVert_{k+1}+\Pr Ra(1 +\kappa^{-1}\lVert \bm{u}\rVert_2)h^k\lVert T\rVert_{k+1} + h^k\lVert p\rVert_k.
	\end{align*}
\end{thm}

\section{Local elimination property and iteration scheme}

\subsection{Local elimination}
In this subsection, we shall show that in the WG scheme (\ref{pb2}), the velocity,   pressure, and temperature approximations, $(\bm{u}_{h0}, p_{h0}, T_{h0})$, defined in the interior of the elements, can be locally eliminated by using the numerical traces, $(\bm{u}_{hb}, p_{hb}, T_{hb})$, defined on the interface of the elements. Therefore, after the local elimination the resultant   system only involves degrees of freedom of $(\bm{u}_{hb}, p_{hb}, T_{hb})$
as unknowns.

We rewrite the   scheme (\ref{pb2}) as the following form:   seek $\bm{u}_h = (\bm{u}_{h0},\bm{u}_{hb})\in \bm{V}_h^0$, $p_h = (p_{h0},p_{hb})\in Q_h^0$ and $T_h = (T_{h0},T_{hb})\in S_h^0$ such that
\begin{align}
\label{pb6}
\left \{
\begin{array}{rl}
A_h(\bm{u}_h;\bm{u}_h,\bm{v}_h)+b_h(\bm{v}_h,p_h)-b_h(\bm{u}_h,q_h)-d_h(T_h,\bm{v}_h)&= (\bm{f},\bm{v}_{h0}), \forall \bm{v}_h\in \bm{V}_h^0,\\
\overline{A}_h(\bm{u}_h;T_h,s_h)&= (g,s_{h0}), \forall s_h\in S_h^0.
\end{array}
\right.
\end{align}
For all $K\in \mathcal{T}_h^f$, taking $\bm{v}_{h0}|_{\mathcal{T}_h^f/K} = \bm{0}, \bm{v}_{hb} = \bm{0}$, $q_{h0}|_{\mathcal{T}_h^f/K} = 0, q_{hb} = 0$ and $T_{h0}|_{\mathcal{T}_h^f/K} = 0, T_{hb} = 0$ in (\ref{pb6}), we can get the following local problem:  seek $(\bm{u}_{h0},p_{h0},T_{h0})\in [P_k(K)]^d\times P_{k-1}(K)\times P_k(K)$ such that, for $\forall (\bm{v}_{h0},q_{h0})\in [P_k(K)]^d\times P_{k-1}(K), s_{h0}\in P_k(K)$, 
\begin{align}
\label{pb7}
\left \{
\begin{array}{rl}
A_{h,K}(\bm{u}_{h0};\bm{u}_{h0},\bm{v}_{h0})+b_{h,K}(\bm{v}_{h0},p_{h0})-b_{h,K}(\bm{u}_{h0},q_{h0})-d_{h,K}(T_{h0},\bm{v}_{h0})&= F_{h,K}(\bm{v}_{h0}),\\
\overline{A}_{h,K}(\bm{u}_{h0},T_{h0};s_{h0})&= G_{h,K}(s_{h0}).
\end{array}
\right.
\end{align}
where 
\begin{align*}
A_{h,K}(\bm{u}_{h0};\bm{u}_{h0},\bm{v}_{h0}) :=& a_{h,K}(\bm{u}_{h0},\bm{v}_{h0}) + c_{h,K}(\bm{u}_{h0};\bm{u}_{h0},\bm{v}_{h0});\\
a_{h,K}(\bm{u}_{h0},\bm{v}_{h0}) :=& \Pr (\nabla_{w,m}\{\bm{u}_{h0},\bm{0}\},\nabla_{w,m}\{\bm{v}_{h0},\bm{0}\}) + \Pr\langle \tau\bm{Q}_l^b\bm{u}_{h0},\bm{Q}_l^b\bm{v}_{h0}\rangle_{\partial K},\\
c_{h,K}(\bm{u}_{h0};\bm{u}_{h0},\bm{v}_{h0}) :=& \frac{1}{2}(\nabla_{w,k}\cdot \{\bm{u}_{h0}\otimes\bm{u}_{h0},\bm{0}\otimes \bm{0}\},\bm{v}_{h0})-\frac{1}{2}(\nabla_{w,k}\cdot \{\bm{v}_{h0}\otimes\bm{u}_{h0},\bm{0}\otimes\bm{0}\},\bm{u}_{h0}),\\
b_{h,K}(\bm{v}_{h0},p_{h0}) :=& (\nabla_{w,k}\{p_{h0},0\},\bm{v}_{h0}),\\
d_{h,K}(T_{h0},\bm{v}_{h0}) :=& \Pr Ra(\bm{j}T_{h0},\bm{v}_{h0}),\\
\overline{A}_{h,K}(\bm{u}_{h0},T_{h0};s_{h0}) :=& \overline{a}_{h,K}(T_{h0},s_{h0}) + \overline{c}_{h,K}(\bm{u}_{h0};T_{h0},s_{h0});\\
\overline{a}_{h,K}(T_{h0},s_{h0}) :=& \kappa (\nabla_{w,m}\{T_{h0},0\},\nabla_{w,m}\{s_{h0},0\}) + \kappa\langle \tau Q_l^b T_{h0},Q_l^bs_{h0}\rangle_{\partial K}\\
\overline{c}_{h,K}(\bm{u}_{h0};T_{h0},s_{h0}) :=& \frac{1}{2}(\nabla_{w,k}\cdot \{\bm{u}_{h0}T_{h0},\bm{0}\},s_{h0})-\frac{1}{2}(\nabla_{w,k}\cdot \{\bm{u}_{h0}s_{h0},\bm{0}\},T_{h0})\\
F_{h,K}(\bm{v}_{h0},q_{h0}) :=& (\bm{f},\bm{v}_{h0}) - \Pr (\nabla_{w,m}\{\bm{0},\bm{u}_{hb}\},\nabla_{w,m}\{\bm{v}_{h0},\bm{0}\}) +\Pr\langle \tau\bm{u}_{hb},\bm{Q}_l^b\bm{v}_{h0}\rangle_{\partial K} \\
& \quad -\frac{1}{2}(\nabla_{w,k}\cdot \{\bm{0}\otimes \bm{0},\bm{u}_{hb}\otimes\bm{u}_{hb}\},\bm{v}_{h0}) - (\nabla_{w,k}\{0,p_{hb}\},\bm{v}_{h0}),\\
G_{h,K}(s_{h0}): =& (g,s_{h0}) - \kappa (\nabla_{w,m}\{0,T_{hb}\},\nabla_{w,m}\{s_{h0},0\}) +\Pr\langle \tau T_{hb},\bm{Q}_l^bs_{h0}\rangle_{\partial K} \\
& \quad -\frac{1}{2}(\nabla_{w,k}\cdot \{\bm{0},\bm{u}_{hb}T_{hb}\},s_{h0}).
\end{align*}

For any $K\in \mathcal{T}_h$, we define the following semi-norms:
\begin{align*}
\interleave \bm{v}_{h0}\interleave_K  :=&\left( \lVert \nabla_{w,m}\{\bm{v}_{h0},\bm{0}\}\rVert_{0,K}^2 + \lVert \tau^{1/2}\bm{Q}_l^b\bm{v}_{h0}\rVert_{0,\partial K}^2\right)^{1/2},\\
\interleave s_{h0}\interleave_K: =&\left( \lVert \nabla_{w,m}\{s_{h0},0\}\rVert_{0,K}^2 + \lVert \tau^{1/2}Q_l^bs_{h0}\rVert_{0,\partial K}^2\right)^{1/2}.
\end{align*}
It is easy to see that the above semi-norms  are norms on the local spaces $[P_k(K)]^d$ and $ P_k(K)$, respectively. 

By following the same routine as  in Section 3 for the global  problem \eqref{pb2}, we can obtain   the following existence and uniqueness results for the local problem \eqref{pb6}.

\begin{thm} For   any given $\bm{u}_{hb}, p_{hb}$ and $T_{hb}$, and any $K\in \mathcal{T}_h$, the local problem (\ref{pb7}) admits at least one solution. In addition, it admits a unique solution if
	\begin{equation*}
	({\Pr}^{-1}\mathcal{N}_{h,K}Ra\kappa^{-1}+\mathcal{M}_{h,K}Ra\kappa^{-2})\lVert G_{h,K}\rVert_h + \mathcal{N}_{h,K}{\Pr}^{-2}\lVert F_{h,K}\rVert_h < 1, 
	\end{equation*}
	where 
	\begin{align*}
\mathcal{N}_{h,K} &:= \sup\limits_{0\neq\bm{w}_{h0},\bm{u}_{h0},\bm{v}_{h0}\in \bm{W}_{h,K}}\frac{c_{h,K}(\bm{w}_{h0};\bm{u}_{h0},\bm{v}_{h0})}{\interleave \bm{w}_{h0}\interleave_K\cdot \interleave \bm{u}_{h0}\interleave_K\cdot\interleave \bm{v}_{h0}\interleave_K},\\
\mathcal{M}_{h,K} &:= \sup\limits_{\substack{0\neq\bm{w}_{h0}\in \bm{W}_{h,K},\\ 0\neq T_{h0},s_{h0}\in P_k(K)}}\frac{\overline{c}_{h,K}(\bm{w}_{h0};T_{h0},s_{h0})}{\interleave \bm{w}_{h0}\interleave_K\cdot \interleave T_{h0}\interleave_K\cdot\interleave s_{h0}\interleave_K},\\
\lVert F_{h,K}\rVert_h &:= \sup\limits_{0\neq\bm{v}_{h0}\in \bm{W}_{h,K}}\frac{F_{h,K}(\bm{v}_{h0})}{\interleave \bm{v}_{h0}\interleave_K},\quad 
\lVert G_{h,K}\rVert_h := \sup\limits_{0\neq s_{h0}\in P_k(K)}\frac{G_{h,K}(s_{h0})}{\interleave s_{h0}\interleave_K},
\end{align*}
and $
\bm{W}_{h,K} := \{\bm{w}_{h0}\in [P_k(K)]^d:b_{h,K}(\bm{w}_{h0},q_{h0}) = 0,\forall q_{h0}\in P_{k-1}(K)\}.$

\end{thm}

\subsection{Iteration scheme}

Since the WG scheme \eqref{pb2} is nonlinear, we introduce the following Oseen's iteration scheme: given   $\bm{u}_h^0$, for $n=1,2,\cdots,$ and $\forall (\bm{v}_h,q_h,s_h)\in \bm{V}_h^0\times Q_h^0 \times S_h^0$, 
\begin{align}
\label{pboseen}
\left \{
\begin{array}{rl}
&a_h(\bm{u}_h^n,\bm{v}_h)+c_h(\bm{u}_h^{n-1};\bm{u}_h^n,\bm{v}_h)+b_h(\bm{v}_h,p_h^n)-b_h(\bm{u}_h^n,q_h)=(\bm{f},\bm{v}_{h0})+d_h(T_h^n,\bm{v}_h),\\
&\overline{a}_h(T_h^n,s_h)+\overline{c}_h(\bm{u}_h^{n-1};T_h^n,s_h)= (g,s_{h0}).
\end{array}
\right.
\end{align}
We have the following convergence theorem.
\begin{thm}
	Let $(\bm{u}_h,p_h,T_h)\in \bm{V}_h^0\times Q_h^0 \times S_h^0$ be the solution to the WG scheme (\ref{pb2}), and  assume that (\ref{c-0}) holds. Then the Oseen's iteration scheme (\ref{pboseen}) is   convergent in the following sense: 
	\begin{align*}
	\lim\limits_{n\longrightarrow \infty}\interleave \bm{u}_h^n- \bm{u}_h\interleave = 0, \lim\limits_{n\longrightarrow \infty}\lVert p_h^n- p_h\rVert = 0, \lim\limits_{n\longrightarrow \infty}\interleave T_h^n- T_h\interleave = 0.
	\end{align*}
\end{thm}
\begin{proof}
	Set $\bm{e}_u^n := \bm{u}_h^n-\bm{u}_h, e_p^n: = p_h^n-p_h, e_T^n: = T_h^n-T_h$, then, from (\ref{pb2}) and (\ref{pboseen}), we have,  for  $\forall (\bm{v}_h,q_h,s_h)\in \bm{V}_h^0\times Q_h^0\times  S_h^0$, 
	\begin{align}\label{pbiteration}
	\left \{
	\begin{array}{rl}
	a_h(\bm{e}_u^n,\bm{v}_h)&=-b_h(\bm{v}_h,e_p^n)+b_h(e_u^n,q_h)+c_h(\bm{u}_h;\bm{u}_h,\bm{v}_h)-c_h(\bm{u}_h^{n-1};\bm{u}_h^n,\bm{v}_h)\\
	\overline{a}_h(e_T^n,s_h)&=\overline{c}_h(\bm{u}_h;T_h,s_h)-\overline{c}_h(\bm{u}_h^{n-1};T_h^n,s_h).
	\end{array}
	\right.
	\end{align}
	Taking $\bm{v}_h = \bm{e}_u^n, q_h = e_p^n, s_h = e_T^n$ in (\ref{pbiteration}), in view of \eqref{3356} and Lemma \ref{boundedness}, we get
	\begin{equation}\label{509}
	\begin{aligned}
	\Pr\interleave \bm{e}_u^n\interleave^2 
	=&c_h(\bm{u}_h;\bm{u}_h,\bm{e}_u^n)-c_h(\bm{u}_h^{n-1};\bm{u}_h^n,\bm{e}_u^n) + d_h(e_T^n,\bm{e}_u^n)\\
	=&-c_h(\bm{e}_u^{n-1};\bm{u}_h,\bm{e}_u^n)-c_h(\bm{u}_h^{n-1};\bm{e}_u^{n},\bm{e}_u^n)+d_h(e_T^n,\bm{e}_u^n)\\
	\leq&\mathcal{N}_h\interleave \bm{e}_u^{n-1}\interleave \interleave \bm{u}_h\interleave \interleave \bm{e}_u^n\interleave+\Pr Ra\interleave e_T^n\interleave\interleave \bm{e}_u^n\interleave,
	\end{aligned}
	\end{equation}
	\begin{equation}\label{510}
	\begin{aligned}
	\kappa\interleave e_T^n\interleave^2=&\overline{c}_h(\bm{u}_h;T_h,e_T^n)-\overline{c}_h(\bm{u}_h^{n-1};T_h^n,e_T^n)\\
	= &-\overline{c}_h(\bm{e}_u^{n-1};T_h,e_T^n)-\overline{c}_h(\bm{u}_h^{n-1};e_T^{n},e_T^n) \\
	\leq& \mathcal{M}_h\interleave \bm{e}_u^{n-1}\interleave \interleave T_h\interleave\interleave \bm{e}_T^n\interleave,
	\end{aligned}
	\end{equation}
	which,   together with \eqref{regular estimates Th}, \eqref{stabs}, and \eqref{c-0},	implies
	\begin{equation*}
	\begin{aligned}
	\interleave \bm{e}_u^n\interleave &\leq {\Pr}^{-1}\mathcal{N}_h\interleave \bm{e}_u^{n-1}\interleave \interleave \bm{u}_h\interleave +Ra\interleave e_T^n\interleave  \\
	&\leq ( {\Pr}^{-1}\mathcal{N}_h \interleave \bm{u}_h\interleave +\kappa^{-1}\mathcal{M}_h\interleave T_h\interleave)\interleave \bm{e}_u^{n-1}\interleave \\
	&\leq (1-C_0)\interleave \bm{e}_u^{n-1}\interleave \leq \cdots 
	\leq (1-C_0)^n\interleave \bm{e}_u^{0}\interleave .
	\end{aligned}
	\end{equation*}
	Since $0<C_0<1$, the above inequality leads to  the conclusion
	\begin{align}\label{ulim}
	\lim\limits_{n\longrightarrow \infty}\interleave \bm{u}_h^n- \bm{u}_h\interleave =\lim\limits_{n\longrightarrow \infty}\interleave\bm{e}_u^n\interleave= 0.
	\end{align}
	Thus, from   \eqref{510}  and \eqref{regular estimates Th} it follows
	\begin{align}\label{Tlim}
	\lim\limits_{n\longrightarrow \infty}\interleave T_h^n- T_h\interleave =\lim\limits_{n\longrightarrow \infty}\interleave e_T^n\interleave= 0.
	\end{align}
	
Finally, in light of  Lemma \ref{infsup} and the first equation of \eqref{pbiteration}, we obtain
\begin{align*}
\lVert e_p^n\rVert \leq &\sup\limits_{\bm{v}_h\in \bm{V}_h^0}\frac{b_h(\bm{v}_h,e_p^n)}{\interleave\bm{v}_h \interleave} \\
= &\sup\limits_{\bm{v}_h\in \bm{V}_h^0}\frac{1}{\interleave\bm{v}_h \interleave}(-a_h(\bm{e}_u^n,\bm{v}_h)+c_h(\bm{u}_h;\bm{u}_h,\bm{v}_h)-c_h(\bm{u}_h^{n-1};\bm{u}_h^n,\bm{v}_h)+d_h(e_T^n,\bm{v}_h))\\
 = &\sup\limits_{\bm{v}_h\in \bm{V}_h^0}\frac{1}{\interleave\bm{v}_h \interleave}(-a_h(\bm{e}_u^n,\bm{v}_h)-c_h(\bm{e}_u^{n-1};\bm{u}_h,\bm{v}_h)-c_h(\bm{e}_u^{n-1};\bm{e}_u^n,\bm{v}_h)-c_h(\bm{e}_u^{n};\bm{u}_h,\bm{v}_h)+d_h(e_T^n,\bm{v}_h))\\
\leq  &\Pr\interleave \bm{e}_u^{n}\interleave +\mathcal{N}_h(\interleave \bm{u}_h\interleave\interleave \bm{e}_u^{n-1}\interleave+\interleave \bm{e}_u^{n}\interleave\interleave \bm{e}_u^{n-1}\interleave+\interleave \bm{u}_h\interleave\interleave \bm{e}_u^{n}\interleave) +\Pr Ra\interleave e_T^n \interleave,
\end{align*}
which, together with (\ref{ulim}) and (\ref{Tlim}), yields 
  $\lim\limits_{n\longrightarrow \infty}\lVert p_h^n- p_h\rVert = 0$.
\end{proof}

\section{Numerical experiments}
In this section, we shall show some numerical results to examine the  performance of the proposed WG   methods for the natural convection equations. The Oseen's iteration scheme \eqref{pboseen} with initial guess $\bm{u}_h^0=0$ is used  in  all the  numerical experiments. 

We consider three cases of our WG methods with $k=1,2$:
\begin{eqnarray*}
\begin{aligned}
WG-I&:& l &= k,&m& = k,\\
WG-II&:& l &= k,&m& = k-1,\\
WG-III&:& l &= k-1,&m& = k-1.
\end{aligned}
\end{eqnarray*}

\begin{exmp} Take   $\Omega = [-1,1]\times[0,1]$ and $\Omega_f = [0,1]\times[0,1]$. 
	The exact solution to the problem \eqref{pb1} is given by
	\begin{align*}
	\left \{
	\begin{array}{rll}
	u_1 &= -x^2(x-1)^2y(y-1)(2y-1) & \text{in} \quad\Omega_f,\\
	u_2 &= y^2(y-1)^2x(x-1)(2x-1) & \text{in}\quad\Omega_f,\\
	p &= x^6-y^6 & \text{in} \quad\Omega_f,\\
	T &= (x-1)(x+1)y(y-1) & \text{in} \quad\Omega.
	\end{array}
	\right.
	\end{align*}
	with $\Pr  = 1,\kappa = 1, Ra = 10$. Regular triangular  meshes are used for the computation (see Figure 1).
\end{exmp}

\begin{figure}[htbp]
	\centering
	\begin{minipage}[t]{0.3\textwidth}		
	\includegraphics[height = 2cm,width=4 cm]{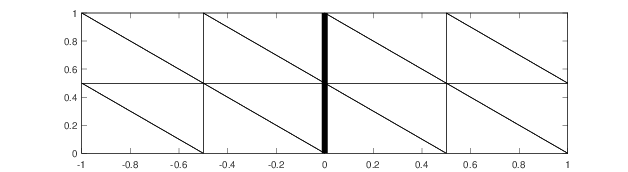} 	
	\end{minipage}
	\begin{minipage}[t]{0.3\textwidth}	
	\includegraphics[height = 2cm,width=4 cm]{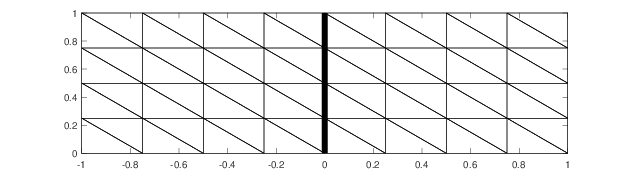}			
	\end{minipage}	
		\caption{Regular triangular meshes:  $4\times 2$ mesh (Left) and $ 8\times 4$ mesh (Right)}	\label{division}
\end{figure}

 Tables 1 and 2 show the history of convergence for the velocity $\bm{u}_{h0} $, pressure $p_{h0}$, and temperature $T_{h0}$. Results of  $div_hU_h =: \max\limits_{K\in \mathcal{T}_h^f}h_K^{-1}\lVert \nabla\cdot \bm{u}_{h0}\rVert_{0,K}$
  are also listed. 
 From the numerical results we have the following observations:
  
  \begin{itemize}
  \item The convergence rates of $\lVert \nabla\bm{u}-\nabla_h\bm{u}_{h0} \rVert_0,$ $ \lVert p-p_{h0} \rVert_0$, and $\lVert \nabla  T-\nabla_hT_{h0} \rVert_0$ for the proposed WG methods  with $k = 1,2$ are  of $k^{\text{th}}$ orders, as is consistent with the theoretical results. In addition, the convergence rates of $\lVert \bm{u}-\bm{u}_{h0}\rVert_0$ and $\lVert T-T_{h0} \rVert_0$ are of $(k+1)^{\text{th}}$ orders.
  
  \item Since $\lVert \nabla_h\cdot \bm{u}_{h0}\rVert_{0,\infty}\apprle \max\limits_{K\in \mathcal{T}_h^f}h_K^{-1}\lVert \nabla\cdot \bm{u}_{h0}\rVert_{0,K}$, the velocity approximations obtained
by our  methods are globally divergence-free, which are conformable to the conclusion in Remark 2.1.
  
  \end{itemize}

\begin{table}[H]
	\normalsize
	\caption{Results for different methods with  $k = 1$ }\label{KK1}
	\centering
\tiny
	\subtable[Method:WG-I]{
	\begin{tabular}{p{0.75cm}<{\centering}|p{1.24cm}<{\centering}|p{0.3cm}<{\centering}|p{1.24cm}<{\centering}|p{0.3cm}<{\centering}|p{1.24cm}<{\centering}|p{0.3cm}<{\centering}|p{1.24cm}<{\centering}|p{0.3cm}<{\centering}|p{1.24cm}<{\centering}|p{0.3cm}<{\centering}|p{1.24cm}<{\centering}}
			\hline   
			\multirow{2}{*}{mesh}&  
			\multicolumn{2}{c|}{ $\frac{\lVert \nabla\bm{u}-\nabla_h\bm{u}_{h0} \rVert_0}{\lVert \nabla\bm{u}\rVert_0}$}&\multicolumn{2}{c|}{$\frac{\lVert \bm{u}-\bm{u}_{h0}\rVert_0}{\lVert \bm{u}\rVert_0}$}&\multicolumn{2}{c|}{ $\frac{\lVert p-p_{h0} \rVert_0}{\lVert p\rVert_0}$}&\multicolumn{2}{c|}{ $\frac{\lVert \nabla  T-\nabla_hT_{h0} \rVert_0}{\lVert \nabla T\rVert_0}$}&\multicolumn{2}{c|}{ $\frac{\lVert T-T_{h0} \rVert_0}{\lVert T\rVert_0}$}
			&\multirow{2}{*}{$div_hU_h$}\cr\cline{2-11}  
			&error&order&error&order&error&order&error&order&error&order&\cr  
			\cline{1-12}
			
			$8\times 4$	    &5.9412E-01    &	    &1.6959E-01	  &       &4.4819E-01   &       &2.4656E-01   &       &2.7341E-02   &      &3.9988E-16\\
			\hline
			$16\times 8$	    &3.1494E-01    &0.92	&4.7778E-02	  &1.83	  &2.3637E-01 	&0.92   &1.2464E-01   &0.98   &6.8747E-03   &1.99  &1.9062E-15\\
			\hline
			$32\times 16$	&1.5988E-01    &0.98	&1.2396E-02	  &1.95	  &1.1983E-01	&0.98   &6.2498E-02   &0.99   &1.7191E-03   &2.00  &3.0715E-15\\
			\hline
			$64\times 32$	&8.0247E-02    &0.99	&3.1249E-03	  &1.99	  &6.0122E-02	&0.99   &3.1272E-02   &1.00   &4.2894E-04   &2.00  &3.1834E-14\\
			\hline
			$128\times 64$	&4.0162E-02    &1.00	&7.8018E-04	  &2.00	  &3.0087E-02	&1.00   &1.5639E-02   &1.00   &1.0704E-04   &2.00  &4.6475E-14\\
			\hline
		\end{tabular}
}
\subtable[Method:WG-II]{
	\begin{tabular}{p{0.75cm}<{\centering}|p{1.24cm}<{\centering}|p{0.3cm}<{\centering}|p{1.24cm}<{\centering}|p{0.3cm}<{\centering}|p{1.24cm}<{\centering}|p{0.3cm}<{\centering}|p{1.24cm}<{\centering}|p{0.3cm}<{\centering}|p{1.24cm}<{\centering}|p{0.3cm}<{\centering}|p{1.24cm}<{\centering}}
		\hline   
		\multirow{2}{*}{mesh}&  
		\multicolumn{2}{c|}{ $\frac{\lVert \nabla\bm{u}-\nabla_h\bm{u}_{h0} \rVert_0}{\lVert \nabla\bm{u}\rVert_0}$}&\multicolumn{2}{c|}{$\frac{\lVert \bm{u}-\bm{u}_{h0}\rVert_0}{\lVert \bm{u}\rVert_0}$}&\multicolumn{2}{c|}{ $\frac{\lVert p-p_{h0} \rVert_0}{\lVert p\rVert_0}$}&\multicolumn{2}{c|}{ $\frac{\lVert \nabla  T-\nabla_hT_{h0} \rVert_0}{\lVert \nabla T\rVert_0}$}&\multicolumn{2}{c|}{ $\frac{\lVert T-T_{h0} \rVert_0}{\lVert T\rVert_0}$}
		&\multirow{2}{*}{$div_hU_h$}\cr\cline{2-11}  
		&error&order&error&order&error&order&error&order&error&order&\cr 
		\cline{1-12}
		
		$8\times 4$	    &7.0486E-01    &	    &7.8104E-01	  &       &4.7353E-01   &       &2.6104E-01   &       &1.3922E-01   &      &1.5492E-15\\
		\hline
		$16\times 8$	    &3.2996E-01    &1.10	&1.8899E-01	  &2.05	  &2.3962E-01 	&0.98   &1.2868E-01   &1.02   &3.5017E-02   &1.99  &5.5321E-16\\
		\hline
		$32\times 16$	&1.6192E-01    &1.03	&4.8031E-02	  &1.98	  &1.2025E-01	&0.99   &6.4066E-02   &1.01   &8.7749E-03   &2.00  &6.8348E-15\\
		\hline
		$64\times 32$	&8.0518E-02    &1.01	&1.2196E-02	  &1.98	  &6.0178E-02	&1.00   &3.1996E-02   &1.00   &2.1989E-03   &2.00  &9.5579E-15\\
		\hline
		$128\times 64$	&4.0158E-02    &1.00	&3.0774E-03	  &1.99	  &3.0095E-02	&1.00   &1.5993E-02   &1.00   &5.5203E-04   &2.00  &2.0390E-14\\
		\hline
	\end{tabular}
}
\subtable[Method:WG-III]{
\begin{tabular}{p{0.75cm}<{\centering}|p{1.24cm}<{\centering}|p{0.3cm}<{\centering}|p{1.24cm}<{\centering}|p{0.3cm}<{\centering}|p{1.24cm}<{\centering}|p{0.3cm}<{\centering}|p{1.24cm}<{\centering}|p{0.3cm}<{\centering}|p{1.24cm}<{\centering}|p{0.3cm}<{\centering}|p{1.24cm}<{\centering}}
		\hline   
		\multirow{2}{*}{mesh}&  
		\multicolumn{2}{c|}{ $\frac{\lVert \nabla\bm{u}-\nabla_h\bm{u}_{h0} \rVert_0}{\lVert \nabla\bm{u}\rVert_0}$}&\multicolumn{2}{c|}{$\frac{\lVert \bm{u}-\bm{u}_{h0}\rVert_0}{\lVert \bm{u}\rVert_0}$}&\multicolumn{2}{c|}{ $\frac{\lVert p-p_{h0} \rVert_0}{\lVert p\rVert_0}$}&\multicolumn{2}{c|}{ $\frac{\lVert \nabla  T-\nabla_hT_{h0} \rVert_0}{\lVert \nabla T\rVert_0}$}&\multicolumn{2}{c|}{ $\frac{\lVert T-T_{h0} \rVert_0}{\lVert T\rVert_0}$}
		&\multirow{2}{*}{$div_hU_h$}\cr\cline{2-11}  
		&error&order&error&order&error&order&error&order&error&order&\cr 
		\cline{1-12}
		
		$8\times 4$	    &7.4774E-01    &	    &8.3792E-01	  &       &4.7910E-01   &       &3.1663E-01   &       &1.6162E-01   &      &1.0304E-16\\
		\hline
		$16\times 8$	    &3.3583E-01    &1.02	&1.9985E-01	  &2.07	  &2.4031E-01 	&0.99   &1.5503E-01   &1.03   &4.0626E-02   &1.99  &2.0466E-16\\
		\hline
		$32\times 16$	&1.6272E-01    &1.01	&5.0551E-02	  &1.98	  &1.2033E-01	&1.00   &7.7080E-02   &1.01   &1.0178E-02   &2.00  &1.7369E-15\\
		\hline
		$64\times 32$	&8.0623E-02    &1.00	&1.2810E-02	  &1.98	  &6.0183E-02	&1.00   &3.8485E-02   &1.00   &2.5494E-03   &2.00  &2.3551E-15\\
		\hline
		$128\times 64$	&4.0212E-02    &1.00	&3.2282E-03	  &1.99	  &3.0093E-02	&1.00   &1.9235E-02   &1.00   &6.3946E-04   &2.00  &6.5944E-15\\
		\hline
	\end{tabular}
}
\end{table}

\renewcommand\arraystretch{1.1}
\begin{table}[H]
	\normalsize
	\caption{Results for different methods with   $k = 2$ }\label{KK2}
	\centering
\tiny
	\subtable[Method:WG-I]{
	\begin{tabular}{p{0.75cm}<{\centering}|p{1.24cm}<{\centering}|p{0.3cm}<{\centering}|p{1.24cm}<{\centering}|p{0.3cm}<{\centering}|p{1.24cm}<{\centering}|p{0.3cm}<{\centering}|p{1.24cm}<{\centering}|p{0.3cm}<{\centering}|p{1.24cm}<{\centering}|p{0.3cm}<{\centering}|p{1.24cm}<{\centering}}
			\hline   
			\multirow{2}{*}{mesh}&  
			\multicolumn{2}{c|}{ $\frac{\lVert \nabla\bm{u}-\nabla_h\bm{u}_{h0} \rVert_0}{\lVert \nabla\bm{u}\rVert_0}$}&\multicolumn{2}{c|}{$\frac{\lVert \bm{u}-\bm{u}_{h0}\rVert_0}{\lVert \bm{u}\rVert_0}$}&\multicolumn{2}{c|}{ $\frac{\lVert p-p_{h0} \rVert_0}{\lVert p\rVert_0}$}&\multicolumn{2}{c|}{ $\frac{\lVert \nabla  T-\nabla_hT_{h0} \rVert_0}{\lVert \nabla T\rVert_0}$}&\multicolumn{2}{c|}{ $\frac{\lVert T-T_{h0} \rVert_0}{\lVert T\rVert_0}$}
			&\multirow{2}{*}{$div_hU_h$}\cr\cline{2-11}  
			&error&order&error&order&error&order&error&order&error&order&\cr 
			\cline{1-12}
			
			$8\times 4$	    &1.6192E-01    &	    &2.8177E-02	  &       &6.6611E-02   &       &2.3814E-02   &       &1.5210E-03   &      &1.5852E-15\\
			\hline
			$16\times 8$	    &4.2800E-02    &1.92	&3.6801E-03	  &2.94	  &1.7476E-02 	&1.92   &5.9899E-03   &1.98   &1.9029E-04   &2.99  &1.3501E-14\\
			\hline
			$32\times 16$	&1.0767E-02    &1.99	&4.6124E-04	  &2.99	  &4.4430E-03	&1.98   &1.4995E-03   &1.99   &2.3790E-05   &3.00  &6.8867E-14\\
			\hline
			$64\times 32$	&2.6808E-03    &2.01	&5.7386E-05	  &3.01	  &1.1115E-03	&1.99   &3.7495E-04   &2.00   &2.9736E-06   &3.00  &3.8064E-14\\
			\hline
			$128\times 64$	&6.7021E-04    &2.00	&7.1513E-06	  &3.00	  &2.7795E-04	&2.00   &9.3738E-05   &2.00   &3.7173E-07   &3.00  &7.2047E-14\\
			\hline
		\end{tabular}
	}
	\subtable[Method:WG-II]{
	\begin{tabular}{p{0.75cm}<{\centering}|p{1.24cm}<{\centering}|p{0.3cm}<{\centering}|p{1.24cm}<{\centering}|p{0.3cm}<{\centering}|p{1.24cm}<{\centering}|p{0.3cm}<{\centering}|p{1.24cm}<{\centering}|p{0.3cm}<{\centering}|p{1.24cm}<{\centering}|p{0.3cm}<{\centering}|p{1.24cm}<{\centering}}
			\hline   
			\multirow{2}{*}{mesh}&  
			\multicolumn{2}{c|}{ $\frac{\lVert \nabla\bm{u}-\nabla_h\bm{u}_{h0} \rVert_0}{\lVert \nabla\bm{u}\rVert_0}$}&\multicolumn{2}{c|}{$\frac{\lVert \bm{u}-\bm{u}_{h0}\rVert_0}{\lVert \bm{u}\rVert_0}$}&\multicolumn{2}{c|}{ $\frac{\lVert p-p_{h0} \rVert_0}{\lVert p\rVert_0}$}&\multicolumn{2}{c|}{ $\frac{\lVert \nabla  T-\nabla_hT_{h0} \rVert_0}{\lVert \nabla T\rVert_0}$}&\multicolumn{2}{c|}{ $\frac{\lVert T-T_{h0} \rVert_0}{\lVert T\rVert_0}$}&\multirow{2}{*}{$div_hU_h$}\cr\cline{2-11}  
		&error&order&error&order&error&order&error&order&error&order&\cr 
			\cline{1-12}
			
			$8\times 4$	    &2.5023E-01    &	    &5.9209E-02	  &       &6.6212E-02   &       &4.1197E-02   &       &4.9610E-03   &      &6.5550E-16\\
			\hline
			$16\times 8$	    &6.3163E-02    &1.98	&7.4474E-03	  &2.99	  &1.7485E-02 	&1.92   &1.0276E-02   &1.99   &6.1111E-04   &3.02  &7.4872E-15\\
			\hline
			$32\times 16$	&1.5659E-02    &2.01	&9.3395E-04	  &2.99	  &4.4432E-03	&1.98   &2.5691E-03   &2.00   &7.5883E-05   &3.01  &5.6488E-15\\
			\hline
			$64\times 32$	&3.8820E-03    &2.01	&1.1720E-04	  &3.00	  &1.1117E-03	&2.00   &6.4257E-04   &2.00   &9.4569E-06   &3.00  &2.4648E-14\\
			\hline
			$128\times 64$	&9.6547E-04    &2.00	&1.4685E-05	  &3.00	  &2.7815E-04	&2.00   &1.6070E-04   &2.00   &1.1804E-06   &3.00  &2.1412E-13\\
			\hline
		\end{tabular}
	}
	\subtable[Method:WG-III]{
	\begin{tabular}{p{0.75cm}<{\centering}|p{1.24cm}<{\centering}|p{0.3cm}<{\centering}|p{1.24cm}<{\centering}|p{0.3cm}<{\centering}|p{1.24cm}<{\centering}|p{0.3cm}<{\centering}|p{1.24cm}<{\centering}|p{0.3cm}<{\centering}|p{1.24cm}<{\centering}|p{0.3cm}<{\centering}|p{1.24cm}<{\centering}}
			\hline   
			\multirow{2}{*}{mesh}&  
			\multicolumn{2}{c|}{ $\frac{\lVert \nabla\bm{u}-\nabla_h\bm{u}_{h0} \rVert_0}{\lVert \nabla\bm{u}\rVert_0}$}&\multicolumn{2}{c|}{$\frac{\lVert \bm{u}-\bm{u}_{h0}\rVert_0}{\lVert \bm{u}\rVert_0}$}&\multicolumn{2}{c|}{ $\frac{\lVert p-p_{h0} \rVert_0}{\lVert p\rVert_0}$}&\multicolumn{2}{c|}{ $\frac{\lVert \nabla  T-\nabla_hT_{h0} \rVert_0}{\lVert \nabla T\rVert_0}$}&\multicolumn{2}{c|}{ $\frac{\lVert T-T_{h0} \rVert_0}{\lVert T\rVert_0}$}
			&\multirow{2}{*}{$div_hU_h$}\cr\cline{2-11}  
			&error&order&error&order&error&order&error&order&error&order&\cr 
			\cline{1-12}
			
			$8\times 4$	    &1.3075E-01    &	    &6.2217E-02	  &       &6.6237E-02   &       &2.1605E-02   &       &5.3332E-03   &      &1.6739E-15\\
			\hline
			$16\times 8$	    &3.4979E-02    &1.90	&7.6750E-03	  &3.02	  &1.7492E-02 	&1.92   &5.4667E-03   &1.98   &6.6033E-04   &3.02  &3.5685E-15\\
			\hline
			$32\times 16$	&8.9627E-03    &1.96	&9.4948E-04	  &3.01	  &4.4333E-03	&1.98   &1.3734E-03   &1.99   &8.2232E-05   &3.01  &3.2603E-14\\
			\hline
			$64\times 32$	&2.2617E-03    &1.99	&1.1834E-04	  &3.00	  &1.1121E-03	&2.00   &3.4409E-04   &2.00   &1.0263E-05   &3.00  &7.3185E-14\\
			\hline
			$128\times 64$	&5.6761E-04    &2.00	&1.4791E-05	  &3.00	  &2.7826E-04	&2.00   &8.6112E-05   &2.00   &1.2821E-06   &3.00  &2.6642E-13\\
			\hline
		\end{tabular}
	}
\end{table}
\begin{figure}[htp]
	\centering

		\includegraphics[height = 4 cm,width=8 cm]{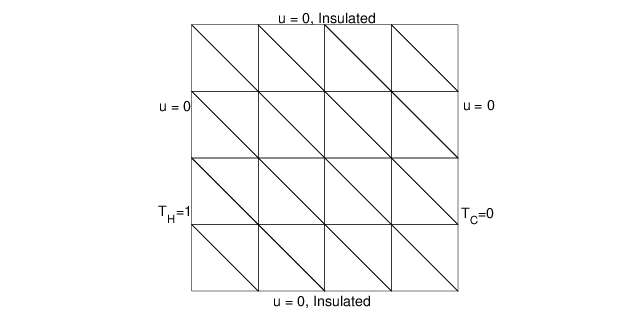} 	

	\caption{The physical domain with its boundary conditions: $4\times4$ mesh}\label{buoyancy}
\end{figure}
 \begin{exmp}
 	We consider   the well-known test cave for the natural convection codes which is called buoyancy-driven cavity problem.  This problem describes 
the two-dimensional flow of a Boussinesq fluid  in an upright square cavity of side $L=1$.   Fig.\ref{buoyancy} shows the physical domain with  the boundary conditions.  The velocity is
zero on all the boundaries. The horizontal walls are insulated with $\frac{\partial T}{\partial \bm{n}} = 0$, and the vertical sides are at
temperatures $T_H=1$ and $T_C=0$. We take  $\Omega=\Omega_f=[0,1]\times [0,1]$,  $\kappa = 1, \Pr = 0.71, \bm{f} = 0$, and $g = 0$.
 \end{exmp}


For different Rayleigh numbers, i.e. $Ra=10^3, 10^4,10^5, 10^6,10^7$,  we use the WG-I method with $k=1,2$ to compute the following quantities at  different mesh sizes: 
\begin{table}[H]
	\normalsize
	\centering{
		\begin{tabular}{c|c p{5cm}|}
			\hline     			
			$u1_{max}$      &the maximum horizontal velocity on the vertical mid-plane of the cavity   \\
			\hline
			$u2_{max}$      &the maximum vertical velocity on the horizontal mid-plane of the cavity   \\
			\hline
			$\overline{Nu}$ &the average Nusselt number throughout the cavity   \\
			\hline
			$Nu_{max}$        &the maximum value of the local Nusselt number on the boundary at x=0   \\
			\hline
			$Nu_{min}$        &the minimum value of the local Nusselt number on the boundary at x=0 \\
			\hline
		\end{tabular}
	}
\end{table}

The results are listed in Table  \ref{num2} and compared with the famous benchmark solutions of de Vahl Davis \cite{de1983natural} and of some other authors such as Manzari \cite{manzari1999explicit}, Massarotti et al \cite{massarotti1998characteristic}, Wan et al \cite{c2001new}, and Zhang et al \cite{Zhang2014Error}. Figure \ref{streamline1}  and Figure \ref{isotherms1} show the contour maps of the stream function and the isotherms of the flow. 
We have the following observations: 
%
\begin{itemize}
\item From Table  \ref{num2} we can see that the WG-I method gives   good  results for all the quantities for different Rayleigh numbers. In particular, the method with $k=2$ behaves very well at the coarsest mesh $40\times40$. 

 \item   Figure \ref{streamline1} demonstrates  that, as Rayleigh number $Ra $ increases, the circular vortex at the cavity center begins to deform into an ellipse and then breaks up into two vortices, and then there's a big vortex in the center. 

\item Figure \ref{isotherms1} shows that, when Rayleigh number $Ra $ is small, the heat transfer mainly depends on  heat conduction (isotherms almost vertical), with the increasing of $Ra$, the heat transfer pattern gradually turns to heat convection and boundary layers appear  around the two walls (isotherms   almost horizontal at the center).
\end{itemize}
\begin{figure}[ht]
	\centering
	\begin{minipage}[t]{0.18\textwidth}		
		\includegraphics[height = 3 cm,width=3 cm]{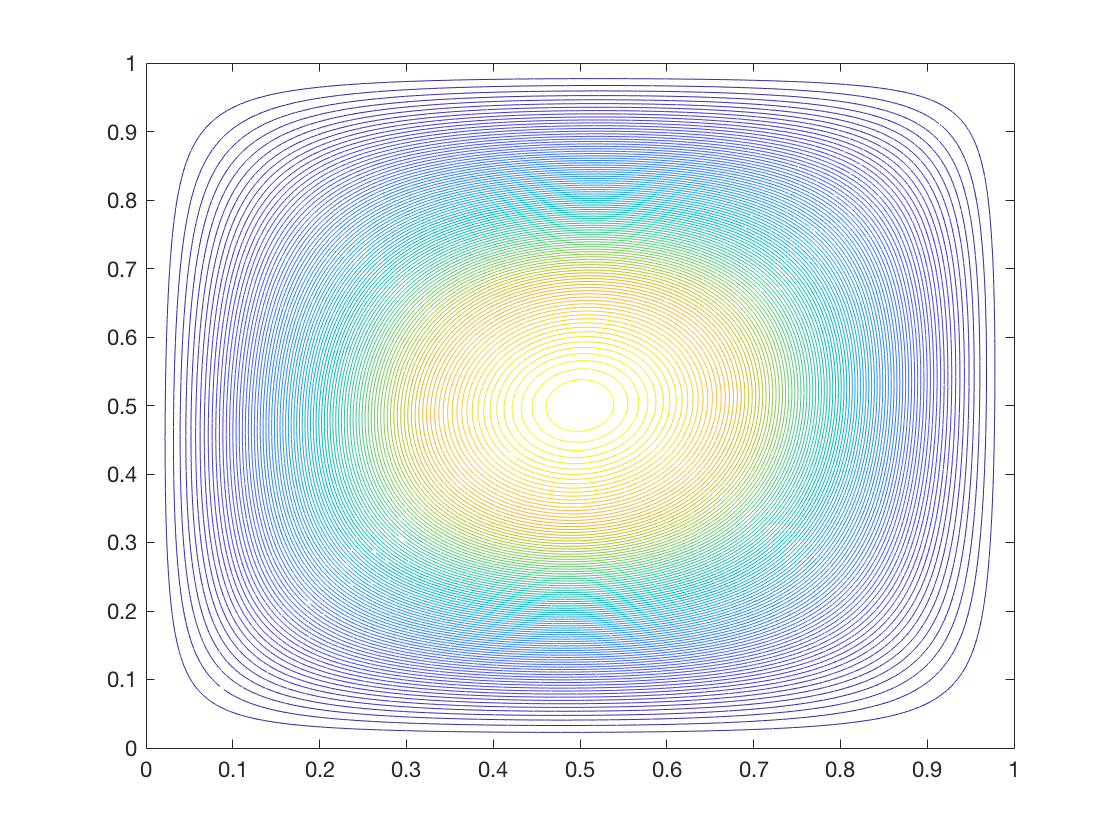} 	
	\end{minipage}
	\begin{minipage}[t]{0.18\textwidth}		
		\includegraphics[height = 3 cm,width=3 cm]{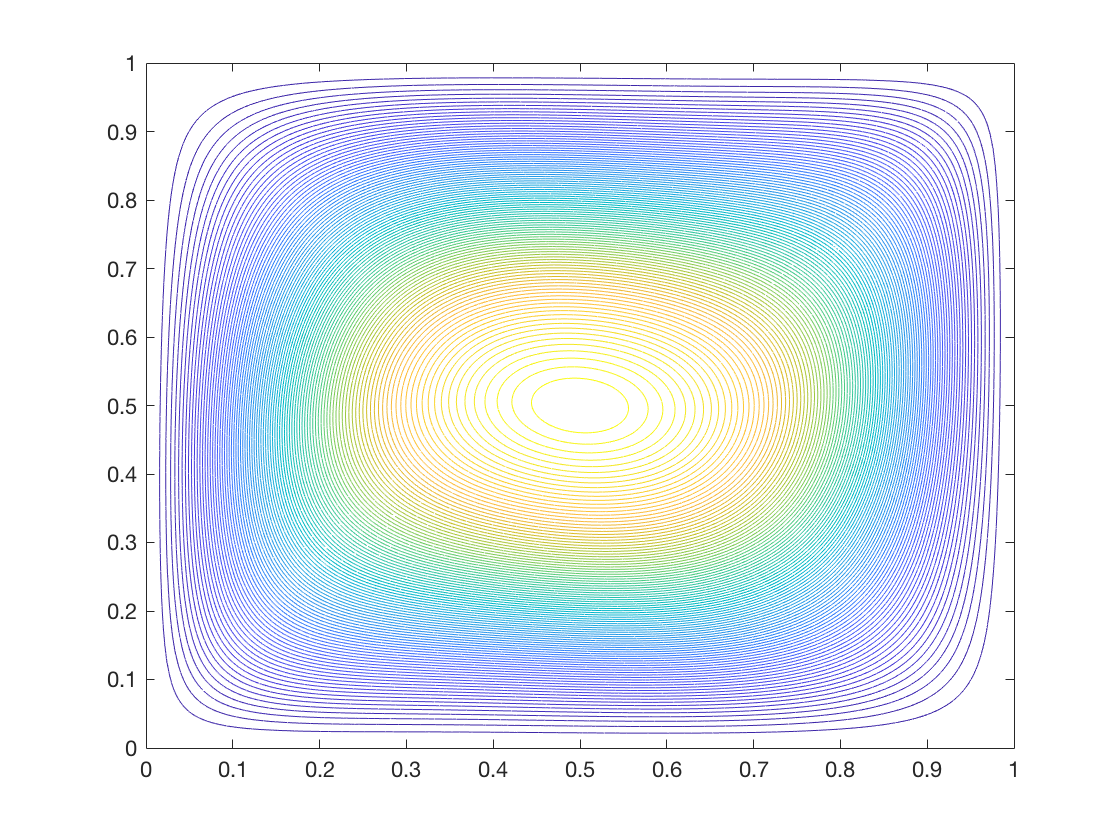} 	
	\end{minipage}
	\begin{minipage}[t]{0.18\textwidth}		
		\includegraphics[height = 3 cm,width=3 cm]{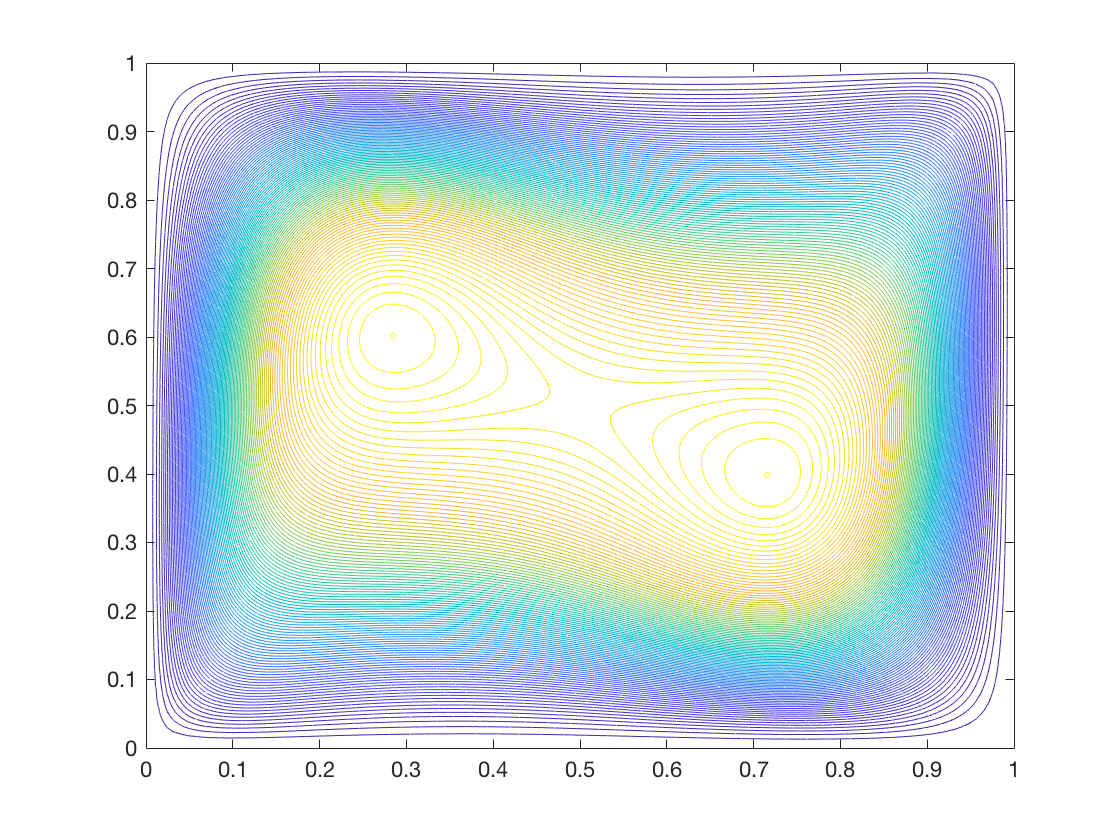} 	
	\end{minipage}
	\begin{minipage}[t]{0.18\textwidth}		
		\includegraphics[height = 3 cm,width=3 cm]{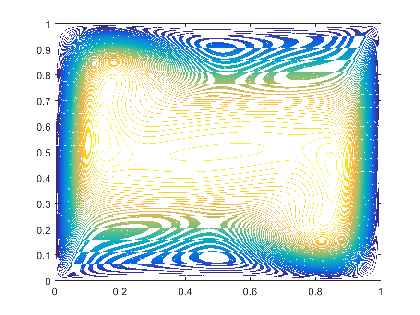} 	
	\end{minipage}
	\begin{minipage}[t]{0.18\textwidth}		
		\includegraphics[height = 3 cm,width=3 cm]{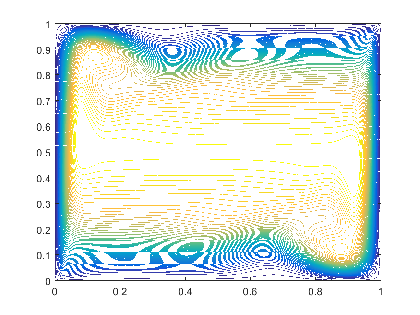} 	
	\end{minipage}
	\tiny\caption{Contour maps of stream function (left to right) with Ra = $10^3, 10^4, 10^5, 10^6, 10^7$.}	\label{streamline1}
\end{figure}

\begin{figure}[ht]
	\centering
	\begin{minipage}[t]{0.18\textwidth}		
		\includegraphics[height = 3 cm,width=3 cm]{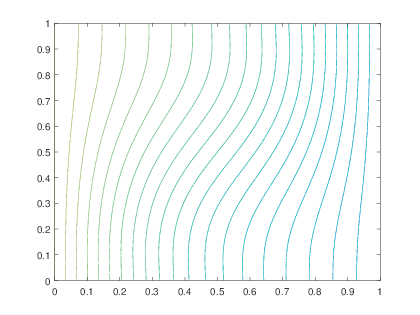} 	
	\end{minipage}
	\begin{minipage}[t]{0.18\textwidth}		
		\includegraphics[height = 3 cm,width=3 cm]{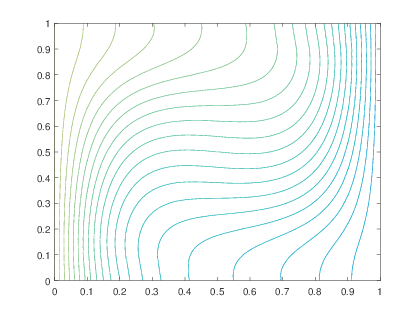} 	
	\end{minipage}
	\begin{minipage}[t]{0.18\textwidth}		
		\includegraphics[height = 3 cm,width=3 cm]{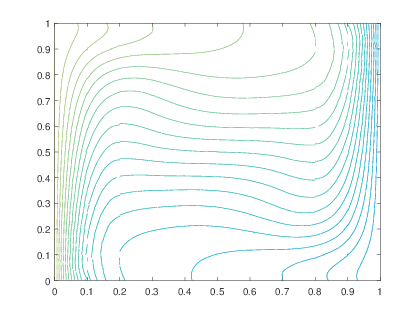} 
	\end{minipage}
	\begin{minipage}[t]{0.18\textwidth}		
		\includegraphics[height = 3 cm,width=3 cm]{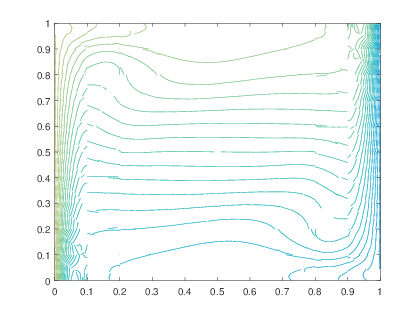} 			
	\end{minipage}
	\begin{minipage}[t]{0.18\textwidth}		
		\includegraphics[height = 3 cm,width=3 cm]{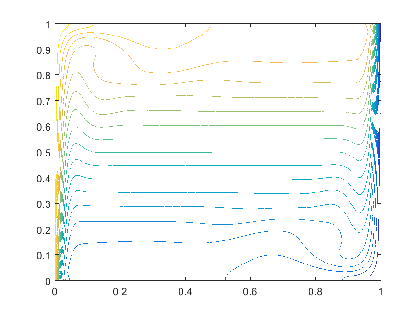} 			
	\end{minipage}
	\tiny\caption{Isotherms(left to right) with Ra = $10^3, 10^4, 10^5, 10^6, 10^7$}	\label{isotherms1}
\end{figure}

\renewcommand\arraystretch{1.1}
\begin{table}[H]
	\normalsize	
	\centering
	\caption{Natural convection in a square cavity: comparison with the benchmark solutions }\label{num2}
	\centering
\scriptsize
	\begin{threeparttable}
	\centering{
	\begin{tabular}{p{0.25cm}<{\centering}|p{0.75cm}<{\centering}|p{1.17cm}<{\centering}|p{1.17cm}<{\centering}|p{1.17cm}<{\centering}|p{1.17cm}<{\centering}|p{0.8cm}<{\centering}|p{0.9cm}<{\centering}|p{0.83cm}<{\centering}|p{0.83cm}<{\centering}|p{1.23cm}<{\centering}}
		\hline   
			\multirow{2}{*}{ Ra}&{}&{WG-I,k=1}&{WG-I,k=1}&{WG-I,k=2}&{WG-I,k=2}&{Ref.\cite{de1983natural}}&{ Ref.\cite{Zhang2014Error}}&{Ref.\cite{manzari1999explicit}}&{ Ref.\cite{massarotti1998characteristic}}&{ Ref.\cite{c2001new}}\cr
		                        &{}&{$40\times 40$}&{$70\times 70$}&{$40\times 40$}&{$50\times 50$}&{}&{ $64\times 64$}&{$70\times 70$}&{$70\times 70$}&{$100\times 100$}\cr
			\cline{1-11}  
			\multirow{5}{*}{$10^3$}&$u1_{max}$	    &3.653    &3.654    &3.640    &3.646    &3.649    &-       &3.68    &-      &3.489          \\
			                       &$u2_{max}$	    &3.711    &3.698    &3.697    &3.697    &3.697 	&-       &3.73    &3.686  &3.69         \\
			                       &$\overline{Nu}$ &1.118    &1.118    &1.118    &1.118    &1.118    &-       &1.074   &1.117  &1.117     \\
			                       &$Nu_{max}$      &1.506    &1.506    &1.560    &1.506    &1.505    &-       &1.47    &-      &1.501    \\
			                       &$Nu_{min}$      &0.691    &0.691    &0.691    &0.691    &0.692    &-       &0.623   &-      &0.691    \\
			\hline
			\multirow{5}{*}{$10^4$}&$u1_{max}$	    &16.227   &16.188   &16.183   &16.180   &16.178	&16.19   &16.10   &-      &16.122     \\
			                       &$u2_{max}$	    &19.744   &19.611   &19.600   &16.628   &19.617	&19.63   &19.90   &19.63  &19.79      \\
			                       &$\overline{Nu}$ &2.243    &2.244    &2.245    &2.245    &2.243    &-       &2.084   &2.243  &2.254    \\
			                       &$Nu_{max}$      &3.528    &3.530    &3.531    &3.531    &3.528    &-       &3.47    &-      &3.579     \\
			                       &$Nu_{min}$      &0.585    &0.585    &0.585    &0.585    &0.586    &-       &0.4968  &-      &0.577    \\
			\hline
			\multirow{5}{*}{$10^5$}&$u1_{max}$	    &34.829   &34.771   &34.715    &34.702  &34.81	&34.74   &34.00   &-      &34.00      \\
			                       &$u2_{max}$	    &69.049   &68.736   &67.875    &68.290  &68.22   &68.48   &70.00   &68.85  &70.63      \\
			                       &$\overline{Nu}$ &4.515    &4.519    &4.522     &4.522   &4.519   &-       &4.30    &4.521  &4.598     \\
			                       &$Nu_{max}$      &7.701    &7.713    &7.716     &7.720   &7.717   &-       &7.71    &-      &7.945    \\
			                       &$Nu_{min}$      &0.726    &0.727    &0.728     &0.728   &0.729   &-       &0.614   &-      &0.698    \\
			\hline
			\multirow{5}{*}{$10^6$}&$u1_{max}$	    &64.977   &64.710   &64.835    &64.541  &64.63	&64.81   &65.40   &-      &65.40      \\
			                       &$u2_{max}$	    &217.307  &221.534  &208.237   &220.609 &219.36  &220.46  &228     &221.6  &227.11    \\
			                       &$\overline{Nu}$ &8.797    &8.813    &8.825     &8.825   &8.800   &-       &8.743   &8.806  &8.976    \\
			                       &$Nu_{max}$      &17.676   &17.511   &17.462    &17.536  &17.925  &-       &17.46   &-      &17.86     \\
			                       &$Nu_{min}$      &0.970    &0.976    &0.980     &0.979   &0.989   &-       &0.716   &-      &0.913     \\
			\hline
			\multirow{5}{*}{$10^7$}&$u1_{max}$	    &154.770  &148.802  &148.454   &148.596 &145.267*&148.40  &139.7   &-      &143.56     \\
			                       &$u2_{max}$	    &819.329  &695.512  &703.702   &707.696 &703.253*&694.14  &698     &702.3  &714.48    \\
			                       &$\overline{Nu}$ &16.564   &16.484   &16.522    &16.521  &-       &-       &13.99   &16.40  &16.656   \\
		                           &$Nu_{max}$      &47.155   &40.374   &40.935    &40.329  &41.025* &-       &30.46   &-      &38.6      \\
			                       &$Nu_{min}$      &1.359    &1.353    &1.363     &1.367   &1.380*  &-       &0.787   &-      &1.298    \\
			\hline

		\end{tabular}
	 \begin{tablenotes}
		\footnotesize
		\item[1] The benchmark solutions with * were mentioned in \cite{Mayne2000h-adaptive} when $Ra=10^7$.
	\end{tablenotes}
	}
\end{threeparttable}
\end{table}

\section{Conclusions}
In this paper, we have developed a class of weak Galerkin finite element methods with globally divergence-free velocity approximation  for the steady-state
natural convection problems.  Well-posedness  of the discrete scheme is analyzed, and  optimal error estimates for the velocity, temperature and pressure approximations are derived. The  proposed Oseen's  iteration algorithm is   unconditionally convergent.  Numerical experiments verify the theoretical results.

\section*{Acknowledgments}
This work was supported by  National Natural Science Foundation of China (11771312) and Major Research Plan of National Natural Science Foundation of China (91430105).


\end{document}